\documentclass[12pt]{article}
\usepackage{psfrag,amsmath,amssymb,latexsym,epic, eepicemu, epsfig,
 float, enumerate}
\usepackage{amscd }
\usepackage{amsfonts}
\usepackage{graphicx}
\usepackage{epsfig,psfrag,amsmath,amssymb,latexsym}
\usepackage{xcolor}
\usepackage{graphics}
\usepackage[T1]{fontenc}
\newtheorem{theorem}{Theorem}[section]

\newtheorem{corollary}{Corollary}[section]

\newtheorem{lemma}{Lemma}[section]

\newtheorem{proposition}{Proposition}[section]
\newtheorem{remark}{Remark}[section]

\newtheorem{assumption}{Assumption}[section]
\newenvironment{proof}[1][Proof]{\textbf{#1.} }{\ \rule{0.5em}{0.5em}}

\newcommand{\esssup}{{\mathrm{esssup}}}
\def\P{{\cal P}}

\newcommand{\bp}{\bar \pi_n}

\newcommand{\ident}{{\mathchoice {\rm 1\mskip-4mu l} {\rm 1\mskip-4mu l}
{\rm 1\mskip-4.5mu l} {\rm 1\mskip-5mu l}}}
\numberwithin{equation}{section}

\newcommand{\X}{{\cal X}}
\newcommand{\B}{{\cal B}}
\newcommand{\e}{\varepsilon}
\def\a{\alpha}
\newcommand{\ba}{\bar{\alpha}}

\newcommand{\x}{{\bf x}}
\newcommand{\p}{{\bar \pi}}

\newcommand{\supp}{\rm{supp\, }}
\newcommand{\F}{{\cal F}}

\title
{An evolution model with uncountably many alleles}

\author{Daniela Bertacchi\\
Dipartimento di Matematica e Applicazioni,\\
Universit\`a di Milano--Bicocca,\\
via Cozzi 53, 20125 Milano, Italy.\\
daniela.bertacchi\@@unimib.it
\and
J\"uri Lember \footnote{Estonian Research Council grant PRG865}\\
Institute of Mathematics and Statistics,\\
University of Tartu,\\
J. Liiv 2, 50409 Tartu, Estonia.\\
juri.lember\@@ut.ee
\and
Fabio Zucca \\
Dipartimento di Matematica,\\
Politecnico di Milano,\\
Piazza Leonardo da Vinci 32, 20133 Milano, Italy.\\
fabio.zucca\@@polimi.it}

\date{}

\begin{document}

\maketitle

\begin{abstract}
We study a class of evolution models, where the breeding process
 involves an arbitrary exchangeable process, allowing for mutations to appear.
The population size $n$ is fixed, hence after breeding, selection is applied. Individuals are
characterized by their genome, picked inside a set $\X$ (which may be uncountable), and there is a fitness
associated to each genome. Being less fit implies a higher chance of being discarded in the
selection process. 
The stationary distribution of the process 
can be described and studied.
We are interested in the asymptotic behavior of this stationary distribution as 
 $n$ goes to infinity.
Choosing a parameter $\lambda>0$ to tune the scaling of the fitness when $n$ grows, we prove limiting theorems
both for the case when the breeding process does not depend on $n$, and for the case when it is given by a Dirichlet process prior.
In both cases, the limit exhibits phase transitions depending on the parameter $\lambda$.
\end{abstract}


\section{Introduction}
\paragraph{The model and setup.} We study the (uncountable) infinite alleles evolution model,
where the breeding and mutation is governed by an $\X$-valued
infinitely exchangeable process $\xi_1,\xi_2,\ldots$, we shall refer to $\xi$ as the breeding process. Throughout the paper $\X$
stands for the Polish (i.e. complete and separable metric) space of all alleles, the elements of $\X$ are
called (geno)types in sequel. By the de Finetti-Hewitt-Savage theorem,
the process $\xi$ is in one-to-one correspondence with a
probability measure $\pi$ on the Borel $\sigma$-algebra of all
probability measures $\P$ equipped with the topology of weak
convergence on $\X$, because for every $n\in \mathbb{N}$, and
$A_i\in \B(\X)$, $i=1,\ldots,n$,
\begin{equation}\label{xi}
{\bf P}(\xi_1\in A_1,\ldots, \xi_n\in A_n)=\int_{\P}\prod_{i=1}^n
q(A_i)\pi(dq),
\end{equation}
where $\B(\X)$ stands for Borel $\sigma$-algebra of $\X$ (see, e.g. \cite[Ch 3]{BNP2} or  \cite[Theorem 1]{cf:Regaz00}). Hence we
identify the process $\xi$ with $\pi$ (see Subsection \ref{sec:pre} for details).
The measure $\pi$ will be referred to as the prior measure. In evolution models theory, the most important and commonly
used prior is the law of Dirichlet process $DP(m,\ba)$, where $\ba$
is a (typically non-atomic) probability measure on $\X$ and $m>0$ is
so-called concentration or precision parameter. With this prior the
breeding process is the following: for every $n\geq 1$, $A\in
\B(\X)$, and population $x_1,\ldots,x_n$
$$P\big(\xi_{n+1}\in A|\xi_1=x_1,\ldots,\xi_n=x_n\big)={m\over
m+n}\ba(A)+{n\over m+n}\cdot {1\over n}\sum_{i=1}^n\delta_{x_i}.$$
If $x^*_1,\ldots,x_k^*$ are the distinct values of $x_1,\ldots,x_n$
with $n_1,\ldots,n_k$ being their frequencies, the conditional
distribution above can be interpreted as follows:
$$\xi_{n+1}|\xi_1=x_1,\ldots,\xi_n=x_n \sim \left\{
                                              \begin{array}{ll}
                                                \ba, & \hbox{with probability ${m\over m+n}$;} \\
                                                \delta_{x^*_j}, & \hbox{with probability ${n_j\over m+n}$\quad $j=1,\ldots,k$.}
                                              \end{array}
                                            \right.$$
(see, e.g. \cite{BNP3,BNP2}). This interpretation allows to obtain the sequence
$\xi_1,\xi_2,\ldots$ by a very simple procedure, known as the
generalized Polya urn scheme. The scheme is very easy to implement
making Dirichlet process priors popular in applications. When $\ba$
is non-atomic, then with probability $m/(m+n)$  the random variable
$\xi_{n+1}$ takes a new value that is not previously seen in
$x_1,\ldots,x_n$ -- a mutation. Hence the ratio $m/(m+n)$ can be
interpreted as the mutation probability. In the literature of
evolution models, the  Polya urn scheme with non-atomic $\ba$ is
often referred to as Hoppe urn (with $m$ typically denoted by
$\theta$), the only difference between
the two urns
is that in the
Polya urn the mutations are labeled and $\ba$ specifies their
origin. In particular, the celebrated Ewens sampling formula, along
with
its consequences, holds under Polya urn scheme, and therefore the
Dirichlet process prior is central in evolution theory.
\\\\
We start with a fixed population size $n$. The process $\xi$ models the breeding: given the
population $x_1,\ldots,x_n$, the new genotype $x_{n+1}$ is bred from the
 conditional
distribution of $\xi_{n+1}$ given $\xi_i=x_i$, $i=1,\ldots,n$, denoted as
$P_{\xi}(\cdot|x_1,\ldots,x_n)$.
Observe that since the order in the population is
arbitrary, exchangeability is a natural assumption about the
breeding process $\xi$. After a new individual with
genotype $x_{n+1}$ is born, it either replaces an already existing
member of the population, or it is discarded, and the population remains
unchanged. The probability that $x_{n+1}$ is kept in the population
depends on the fitnesses of all population members. So, in what
follows, let $w:\X\to \mathbb{R}^+$ be a bounded continuous strictly
positive function, assigning a fitness to every type.  The bigger
$w(x_{n+1})$, the more likely that a newborn member replaces an
already existing one in the population. There are several selection
schemes possible. In
\cite{article1}, the following schemes were introduced.\\\\
{\bf Single tournament selection:}
\begin{enumerate}
\item Sample $ x_{n+1} \sim  P_\xi( \cdot \mid x_1, \ldots, x_n ) $
\item Sample $i$ randomly from $\{ 1, \ldots, n\}$
\item With probability $\frac{w(x_{n+1})}{w(x_i) + w(x_{n+1})}$ replace $x_i$ with $x_{n+1}$ and discard $x_i$, otherwise discard $x_{n+1}$.
\end{enumerate}
{\bf Inverse fitness selection:}
\begin{enumerate}
\item Sample $ x_{n+1} \sim  P_\xi( \cdot \mid x_1, \ldots, x_n ) $
\item Sample $i$ from  $\{1,\ldots,n+1\}$ with probabilities proportional to $\{ \frac{1}{w(x_1)}, \ldots , \frac{1}{w(x_{n+1})}\}$
\item If $i<n+1$, then replace $x_i$ by $x_{n+1}$.
\end{enumerate}
Both
selection schemes define a Markov kernel on $\X^n$. Lemma
\ref{lemma:rev} below proves that both kernels satisfy the detailed
balance equation with stationary measure
\begin{equation}\label{pn} P_n(A):={1\over Z_n}\int_A \prod_{j=1}^n
w(x_j)P^n_{\xi}(d\x),\quad A\in \B(\X^n),
\end{equation}
where $P^n_{\xi}$ is the law of $(\xi_1,\ldots,\xi_n)$,
$\x=(x_1,\ldots,x_n)\in \X^n$ and $Z_n$ is the normalizing constant.
Hence the stationary (or limit) distribution of the genotypes in
$n$-elemental population has clear and explicit closed form,
depending solely on
$w$ and $\pi$.\\\\
The measure $P_n$ in \eqref{pn} is the main object of interest. The article focuses on
 the limit of $P_n$ when the population size $n$ grows
and the fitness function $w_n$ and prior $\pi_n$ both might depend on $n$. As
such, the question is incorrect, because $P_n$ is defined on
different domains $\X^n$. To overcome that problem, we consider two
approaches in parallel:
\begin{itemize}
                         \item The first approach is to consider the triangular array of random
variables
$$(X_{1,n},\ldots,X_{n,n})\sim P_n,\quad n=1,2,\ldots $$
and ask: is there a limit stochastic process $X_1,X_2,\ldots$ such that
for every $m$ and for every $m$-tuple of integers
$t_1,t_2,\ldots,t_m$, it holds (as $n\to \infty$) that
\begin{equation}\label{fi-con}
(X_{t_1,n},X_{t_2,n},\ldots,X_{t_m,n})\Rightarrow
(X_{t_1},X_{t_2},\ldots,X_{t_m}).\end{equation}
 If such a limit
process exists, it can be considered as an approximation of the
population $(X_{1,n},\ldots,X_{n,n})$ for big $n$. Theorem
\ref{thmP} provides the main
technical tool for proving the convergence \eqref{fi-con} and,
hence, the existence of the limit
process.
                         \item The second approach is to transfer the measures $P_n$
into the measures $Q_n$ on $\P$. For that we define the mapping $g$ that
maps a vector $\x$ to its empirical measure:
\begin{equation}\label{defg}
g:\X^n\to \P,\quad g(\x)={1\over n}(\delta_{x_1}+\cdots
+\delta_{x_n})
\end{equation}
and we define  $Q_n$ as $P_ng^{-1}$, i.e. the pushforward measure
\begin{equation}\label{eq:Qn}
Q_n(E):=P_n\big(g^{-1}(E)\big),\quad E\in \B(\P).
\end{equation}
In other words, $Q_n$ is the distribution of $g(X_1,\ldots,X_n)$,
where $(X_1,\ldots,X_n)\sim P_n$. Since $P_n$ is invariant with
respect to permutations, i.e. the $n$-dimensional random vector
having distribution as $P_n$ is exchangeable, we see that $P_n$ can
be uniquely restored from $Q_n$, so in a sense they are the same. In
Subsection \ref{sec:PandQ}, we shall argue that $g$ is measurable, i.e.
$Q_n$ is well defined. Since the measures $Q_n$ are defined on the
same domain $\P$, the question now is the existence of a limit measure
$Q^*$ such that $Q_n\Rightarrow Q^*$  (the weak convergence). The
$Q_n$-counterpart of Theorem \ref{thmP} is Theorem \ref{thmQ} that
provides necessary conditions  in terms of $w_n$ and $\pi_n$ for
existence of $Q^*$. Since the proof of Theorem \ref{thmQ} is based
on large deviation inequality, an additional assumption that $\X$ is
compact is imposed.
                       \end{itemize}
\paragraph{The phase transition results.} For the first convergence
results in Section \ref{sec:pi} we consider the case where the prior is arbitrary and independent of $n$, $\pi_n=:\pi$.
As weight functions, we take
\begin{equation}\label{wn}
w_n(x)=\exp\left[-{\phi(x)\over n^{\lambda}}\right],
\end{equation}
 where $\lambda\geq 0$
and the function $\phi(x)$ is nonnegative, continuous  and bounded.
The parameter $\lambda$ controls how fast the differences between
fitness functions $w_n$ vary when $n$ increases. The case
$\lambda=0$ corresponds to the case $w_n=w$ for every $n$.
Observe that $P_n$ in \eqref{pn} is invariant with respect to multiplying $w$ by a positive constant, hence when $w$ is bounded from above, there is no loss of generality in taking it to be bounded by 1 as \eqref{wn} implies.
Let us note that $\lambda=1$ can be encountered in the literature, at least in the two-allele model
($|\X|=2$). Indeed often, in that case, fitness is taken as $1$ and $(1+s)$ where
$s\cdot n\to \gamma$, so $(1+s)\approx \exp[\gamma/n]$, which corresponds to
$\lambda=1$. We shall see from the limiting results as $n$ goes to infinity,
that $\lambda=1$ is, in a sense, the right scaling. \\\\
Theorems \ref{thm1} and \ref{thm12} are the main phase transition
theorems for arbitrary $\pi$. The results of both theorems
can summarized as follows:
\begin{description}
  \item[The case $\lambda>1$:] Then \eqref{fi-con} holds with the limit
  process being equal to the breeding process $\xi$ and $Q_n\Rightarrow \pi$. This means that
  when $\lambda>1$, then the differences between
  fittnesses vanish so quickly that  the selection has
  no influence in the limit.
  \item[The case $\lambda=1$:] Then \eqref{fi-con} holds with the limit
  process being an infinitely exchangeable process with prior
  measure ${\bar \pi}$ (specified in Theorem \ref{thm1}) that depends on $\phi$ and $\pi$ and is
  different from $\pi$. Then also $Q_n\Rightarrow \p.$
  \item[The case $\lambda\in [0,1)$:] In this case we impose an additional mild
  assumption on $\phi$ (that in particular guarantees the uniquess of the minimum $x_o$), and
  we also assume that the support of $\pi$ contains $\delta_{x_o}$.
  Then
  \eqref{fi-con} holds with the limit
  process being degenerate with one possible path $x_o,x_o,\cdots$
  and $Q_n\Rightarrow \delta_{q^*}$, where $q^*=\delta_{x_o}$. This
  means that when $\lambda<1$ then the selection is so strong that breeding has no influence in the limit and only the
  fittest type $x_o$ (that maximizes $w_n$ for every $n$) survives.
\end{description}
In Section \ref{sec:dir} we consider the case when the prior measure is
the law of DP($m_n,\ba$), but the fitness function is still \eqref{wn}.
We let the concentration parameter $m_n$ depend on the
same parameter $\lambda$ as follows: $m_n=cn^{1-\lambda}$, where
$c>0$. When $\lambda=1$, then $m_n=c$, and therefore $\pi$  is
independent of $n$, hence this case is the  case
considered above. However, the case $\lambda\in [0,1)$ needs special
treatment. Observe that when $\lambda=0$, then $w_n=w$ and the
mutation probability $m_n/(n+m_n)=c/(1+c)$ is independent of the
population size
$n$ and this makes the case $\lambda=0$ appealing. \\\\
The results with Dirichlet process prior, Theorems  \ref{thm1a} and \ref{thm1b} can be summarized as follows,
the additional assumptions are that  $\X$ is compact, $\ba$ has full support  and $x_o$ is the unique minimizer of $\phi$.
\begin{description}
  \item[The case $\lambda=0$:] Then \eqref{fi-con} holds with the limit
  process being an iid process with  $X_i\sim r^*$ and  $Q_n\Rightarrow \delta_{r^*}$. The measure $r^*$ depends on the inequality
\begin{equation}\label{marta0}
\int_\X {{w}(x_o)\over w(x_o)-w(x)}\ba(dx)\geq {1+c\over c}.
\end{equation}
When \eqref{marta0} holds, then $r^*$ has  density $r^{*}(x)$
with respect to $\ba$:
$$r^{*}(x)={c w(x)\over \theta(1+c)-w(x)},$$
where
$\theta>0$ depends on $w$, $\ba$ and $c$ (see Lemma \ref{lemma1b}).
 Observe that the density is with respect to $\ba$-measure, so  when $\ba$ has an atom, then  $r^*$ (i.e. the limit population) has the same atom, but its mass is
re-weighted. But when \eqref{marta0} fails, then
$r^{*}$ has an absolutely continuous part (with respect to $\ba$) with density
$$r_a^{*}(x)={cw(x)\over w(x_o)-w(x)}$$
but also an atom at $x_o$. When \eqref{marta0} fails, then $\ba$ cannot have an atom at $x_o$, but $w$ is peaked so heavily in neighborhood of $x_o$ that in the limit measure an atom appears. Thus, when \eqref{marta0} fails, then in the limit population there is a fixed proportion of individuals with the fittest type $x_o$.

  \item[The case $\lambda\in (0,1)$:] Then \eqref{fi-con} holds with the limit
  process being an iid process with  $X_i\sim q^*$ and $Q_n\Rightarrow q^*$. The measure $q^*$ depends on the inequality
  \begin{equation}\label{marta20}
\int_\X {1\over \phi(x)-\phi(x_o)}\ba (dx)\geq {1 \over c}.
\end{equation}
When \eqref{marta20} holds, then $q^*$ has density $q^{*}(x)$
with respect to $\ba$:
$$q^{*}(x)={c\over \phi(x)-c-\theta},$$
where $\theta>0$ is a parameter. When \eqref{marta20} fails, then  $q^*$ has an absolutely continuous part
(with respect to $\ba$) with density
$$q_a^{*}(x)={c\over \phi(x)-\phi(x_o)}$$
but also an atom at $x_o$. Thus, when \eqref{marta20} fails, hence $\ba(x_o)=0$, in the limit an atom at $x_o$ is created.
\end{description}
What is remarkable that in the last case the limit measure is independent of $\lambda\in (0,1)$.
\paragraph{The case of finite $\X$ and the relation with previous work.}
The literature on mathematical population genetics is vast and focuses on various aspects of the evolution of
traits within a population (see \cite{cf:ewens} and references therein). A common feature of these models is the fact that
individuals, characterized by their genome (or ``type'', ``trait''), breed and die. Mutations may occur and a fitness be associated to each type. The population can be modelled as having a varying or a fixed size. In the first case the process is usually a birth-death process and one can focus either on the equilibrium when time grows, or on the trajectories. For instance the asymptotic distribution of fitnesses
is studied by \cite{cf:Schi, cf:ben-ari, cf:BLZ} in an evolution scheme where a random number of least fit individuals die at each generation, while \cite{cf:BZA, cf:IL} study the effect of random/deterministic events on this distribution;
\cite{cf:BCMel} and \cite{cf:BSW21} study the convergence of the evolutionary process as the population size goes to infinity.
The population is assumed to have a constant size in classical models such as the Wright-Fisher and the Moran models, but
even with this assumptions there are still many theoretical challenges and applications (see for instance
\cite{cf:Durrett-mayberry, cf:Schw1, cf:Schw2}).\\
In our model, the population has constant size $n$, breeding is conditional sampling from a very general exchangeable process
and selection takes place at death (where less fit individuals are more likely to be removed). We are interested in the asymptotic
behaviour, as the size of the population increases, of the stationary distribution of types.\\\\
The case of finitely many types $|\X|=K<\infty$ was treated in \cite{article1}. Then $\P$ is just a simplex and when $\pi$ is the Dirichlet distribution ${\rm Dir}(\a_1,\ldots,\a_K)$, then the breeding process can be considered
as a version well known Moran model  without selection (see \cite[Sec 2.1]{article1}).  When the selection (either single tournament or inverse fitness) is added, then we end up with a version of Moran model with breeding and selection. The stationary measure $P_n$ in this case ($\pi$ equals to Dirichlet distribution) in terms of allele counts is as follows
\begin{equation}\label{cnts2}
P_{n}(n_1,\ldots,n_K)={1\over Z_n}{n!\over n_1!\cdots n_K!}{(\alpha_1)_{n_1}\cdots (\alpha_K)_{n_K}\over (|\alpha|)_n}w^{n_1}(1)\cdots w^{n_K}(K),
\end{equation}
where $(\a)_{n}=\alpha(\a+1)\cdots(a+n-1)$, $|\a|=\a_1+\cdots+\a_K$ and $n_k\geq 0$ stands for the number of type $k$ in $(x_1,\ldots,x_n)$, thus $n_1+\cdots+n_K=n$. Since $P_n$ is exchangeable, it can be equally presented in terms of counts and in the literature it is typically done so. As pointed out in \cite{article1} there are many other versions of Moran models leading the same stationary distribution (the parameters $\a$ are obtained then from mutation probabilities). Hence the case $\pi={\rm Dir}(\a_1,\ldots,\a_K)$ is important special case also in finite allele model.
\\\\
Although one may argue that in reality the set of types $\X$ is always finite, we emphasize that its cardinality is much larger than the
population size itself. Thus, since we are considering limits when the population size grows to infinity, it is reasonable to assume that
$\X$ is infinite, which poses some difficulties in the treating.
For finite $\X$, the limit theorems of this paper hold true as well (they are just a special case), but the proofs are much simpler. Hence in a sense the current article can be considered as a generalization of  \cite{article1}, but the generalization is far from being trivial. For general $\X$, a new machinery needs to be built, and it is the purpose of the current paper. The difference between general and finite $\X$ is well illustrated by the Dirichlet process results (Theorems \ref{thm1a} and \ref{thm1b}). In finite case, since $\alpha_k>0$ for every $k=1,\ldots,K$, the equalities  \eqref{marta0} and \eqref{marta20} both hold. In the finite case the limit probability measures
$r^*$ and $q^*$ are elements of simplex (thus $K$-dimensional vectors) as follows:
$$r^*(k)={w(k)\alpha_k\over \theta(1+|\a|)-w(k)},\quad q^*(k)={\alpha_k\over\phi(k)+|\a|-\theta},\quad k=1,\ldots,K$$
where in both formulas $\theta$ is a parameter. Since in the
 finite case $c=|\a|$ and $\ba_k={\alpha_k\over |\a|}$, we see that these measures are indeed the same as given by Theorems \ref{thm1a} and \ref{thm1b}. However, quite surprisingly for general $\X$ the additional atom appears. This is something one cannot predict based solely on the results of \cite{article1}. Also the proofs  Theorems \ref{thm1} and \ref{thm12} for general $\X$ are essentially different from the ones in the case of finite $\X$, they are based on large deviation result and therefore the additional assumption of compactness is needed.  We also would like to stress out that all limits in \cite{article1} as  well as all limits in the in the current paper are obtained without diffusion approximation. For example, when $\X$ is finite, $\pi$ is Dirichlet distribution and
 $\lambda=1$, the the limit density -- sometimes called as Wright's formula -- can be found in the literature
(for references, see \cite[Sec 3.4]{article1}), typically connected to the diffusion approximation. However,  the proof in \cite{article1} uses fairly simple mathematics and no diffusion approximation, the generalization to general $\X$ (Theorem \ref{thm1} in the current article) uses more involved mathematics, but again no diffusion approximation.
\paragraph{Outline of the paper.} In Section \ref{sec:pre}, the model and the main objects of the article,
are formally defined. In Subsection \ref{subsec:det}, the detailed balance equation is proven showing that $P_n$ is indeed the stationary measure of the model (with population size $n$).
In Subsection \ref{sec:PandQ} we give  an alternative representation of $P_n$ and define the measure $Q_n$.
In Section \ref{sec:limit}, the limit process (infinite population) and the sense of convergence are defined. The main results of Section \ref{sec:limit} are Theorems \ref{thmP} and \ref{thmQ}. The first of them proves the existence of the limit process under rather general assumptions and the second theorem shows there also exists a limit measure $Q$ such that $Q_n \Rightarrow Q$ (under the additional assumption that $\X$ is compact). These theorems are the basis of the  paper. Section \ref{sec:pi} is devoted to the case when the prior measure $\pi$ is arbitrary but independent of $n$. The main results of that section are Theorems \ref{thm1} and \ref{thm12}; these two theorems together  give the first  phase transition result as described in Introduction. In Section \ref{sec:dir}, the Dirichlet process prior is considered. The main results are Theorem \ref{thm1a} and \ref{thm1b} which provide phase transition results for that case. The proofs of these theorems are rather technical and therefore they are presented in Appendix.
\section{Preliminaries}\label{sec:pre}
Recall that ${\cal P }$ stands for the set of all probability measures on Borel $\sigma$-algebra ${\cal B}(\X)$. In what follows, we shall denote the elements of ${\cal P}$ by
 $q$. For any integrable function $f$ on $\X$, we shall denote by
 $$\langle f,q \rangle:=\int_\X f(x)q(dx).$$
 The set ${\cal P }$ is equipped with Prokhorov metric and so this is a complete separable metric
 space as well, see (\cite[p. 72]{Billingsley}). Prokhorov metric
 metrizes the weak convergence of probability measures, denoted by $q_n\Rightarrow q$ in the sequel, and the Borel
 $\sigma$-algebra   ${\cal B}({\cal P })$
 is such that for any $A\subset {\cal B}(\X)$, the mapping $q\mapsto
 q(A)$ is ${\cal B}({\cal P })$-measurable  (see e.g.
\cite[Prop A.5]{BNP3}). Also for any continuous bounded function $f$ on
 $\X$, the function $q\mapsto \langle f,q \rangle$ is continuous and hence ${\cal B}({\cal P })$-measurable as
 well. In what follows, we shall also see $n$-fold product measures
$q^n$ on $\B(\X^n)$. Since $q\mapsto q(A)$ is measurable for every
$A\in \B(\X)$, then also $q\mapsto q^n(A)$ is measurable for every
$A\in \B(\X^n)$. This follows from Dynkin's $\pi-\lambda$ theorem:
clearly for any cylinder $A=A_1\times \cdots \times A_n$ the mapping
$q\mapsto q^n(A)=q(A_1)\cdots q(A_n)$ is measurable (as a product of
measurable functions). The set $\Lambda=\{A\in \B(\X^n): q\mapsto
q^n(A) \text{  is measurable  }\}$ is a $\lambda$-system (i.e.
contains $\X^n$ and closed with respect to complements and disjoint
unions) containing all cylinders. Since $\X$ is Polish, the
cylinders generate $\B(\X^n)$, and so by Dynkin's $\pi-\lambda$
theorem $\B(\X^n)\subset \Lambda$.\\\\
To see the one-to-one correspondence between breeding process $\xi$ and measures $\pi$, observe that for every measure $\pi$, there is a process $\xi$ satisfying equation~\eqref{xi}. Indeed, for
any $A\in \B(\X^n)$,
the map $q\mapsto q^n(A)$ is integrable then
${P}_{\xi}^n(A)=\int_\X
 q^n(A)\pi(dq)$ exists and the family $\{{P}_{\xi}^n\}$ satisfies
Kolmogorov's consistency conditions. The claim follows from
\cite[Theorems 12.7 and 15.26]{cf:Alip}. The other direction -- to every  exchangeable $\xi$ there corresponds a measure $\pi$ -- follows from Finetti-Hewitt-Savage representation. Recall that
$P_{\xi}(\cdot|\x)$ stands for the
 conditional distribution of $\xi_{n+1}$. The existence of a regular version for the conditional probability in Polish spaces is a consequence of
\cite[Theorem 7.8]{cf:PP1966}. Indeed it is enough that the $\sigma$-algebra contains a sub-$\sigma$-algebra which is separable (generated by a countable collection of sets) and the probability measure is compact approximable. Both conditions hold for a probability measure on the Borel $\sigma$-algebra of a Polish metric space: the first one is trivial and the second one follows from
\cite[Theorem 11.20]{cf:Alip}.
\subsection{Detailed balance equation}\label{subsec:det}
\paragraph{Kernels.} Recall the two selection schemes: single tournament and inverse fitness. Both define a Markov chain with uncountable state space $\X^n$.
When $\X$ is finite, as in \cite{article1}, then the corresponding
transition matrix is easy to define. We now define the corresponding
transition kernels for both schemes. Recall that $w$ is a strictly positive, bounded and continuous fitness function on $\X$.\\\\
Let, for every $A\in \B(\X^n)$, $\x\in \X^n$ and $k=1,\ldots,n$
$$A_k(\x):=\{x\in \X: (x_1,\ldots, x_{k-1},x,x_{k+1},\ldots, x_n)\in
A\}.$$
Observe that when $A=A_1\times \cdots \times A_n$ is a cylindrical
set, then
$$A_k(\x)=\left\{
            \begin{array}{ll}
              A_k, & \hbox{when $x_j\in A_j$, for every $j\in \{1,\ldots,k-1,k+1,\ldots,n\}$;} \\
              \emptyset, & \hbox{else.}
            \end{array}
          \right.
$$
The transition kernel corresponding to the single tournament selection
is
\begin{align*}
P(\x,A)=&{1\over n}\sum_{k=1}^n P_k(\x,A),\quad {\rm where}\\
P_k(\x,A)
:=&\int_{A_k(\x)}{w(x)\over w(x)+w(x_k)}
P_{\xi}(dx|\x)+ b_k(\x)\delta_{x_k}(A_k(\x)),\\
b_k(\x):=&\int_{\X}{w(x_k)\over w(x)+w(x_k)}P_{\xi}(dx|\x).
\end{align*}
Here $b_k(\x)$ is the probability that a newborn individual $x_{n+1}$
looses the tournament to $x_k$. Hence, the first term of $P_k(\x,A)$ is
the probability that $x_{n+1}$ wins over $x_k$ and is born in
$A_k(\x)$; the second term is the probability that $x_{n+1}$ looses
the tournament to $x_k$. Clearly $P_k(\x,\cdot)$ a probability
measure on $\B(\X^n)$. The weight $n^{-1}$ represents the fact that
all individuals in population have equal probability to be picked
for the tournament. We observe that, applying Dynkin Theorem, one can prove that
$x\mapsto P_k(\x,A)$ is
measurable for every $A\in \B(\X^n)$.\\\\
The transition kernel corresponding to the inverse fitness selection
is
\[
\begin{split}
 \widetilde P(\x, A)&=\sum_{k=1}^n \widetilde P_k(\x,A)+
  c(\x)\delta_\x(A), \quad \text{where}\\
 \widetilde P_k(\x,A) &:=
 \int_{A_k(\x)}
 \frac{1}{\sum_{j=1}^{n+1} w(x_k)/w(x_j)}
P_\xi(dx_{n+1}| \x)\\
  c(\x) &:= \int_{\mathcal{X}}
  \frac{1}{\sum_{j=1}^{n+1} w(x_{n+1})/w(x_j)}
P_\xi(dx_{n+1}| \x).
   \end{split}
\]
Here $c(\x)$ is the probability that $x_{n+1}$ is chosen and so nothing is changed, the first term in  $\widetilde P_k(\x,A)$ is the probability that $x_k$ is chosen and newborn $x_{n+1}$ is in $A_k(\x)$.
\paragraph{Reversibility.}
The following lemma shows
the $P_n$, defined in \eqref{pn}, is the stationary measure for both  single tournament and inverse fitness kernel,
and the stationary process is reversible.
\begin{lemma}\label{lemma:rev} Let $P(\x,A)$ be the transition kernel corresponding
to the single tournament selection (resp.~to the inverse fitness selection). Then, for every $B,A\in \B(\X^n)$,
it holds
\begin{equation}\label{dbe}
\int_B P(\x,A)P_n(d\x)=\int_A P(\x,B)P_n(d\x).
\end{equation}
\end{lemma}
\begin{proof} It suffices to prove \eqref{dbe} if $A$ and $B$ are
both cylinders: $A=A_1\times \cdots \times A_n$, $B=B_1\times \cdots
\times B_n$.\\
Let us consider the single tournament selection kernel. For any fixed $k$
\begin{align*}
&\int_B P_k(\x,A
)P_n(d\x)=\\
&\frac{1}{Z_n}\int_B \int_{A_k(\x)}{w(x_{n+1})\over
w(x_{n+1})+w(x_k)}P_{\xi}(dx_{n+1}|\x)w(x_1)\cdots w(x_n)P^n_{\xi}(d\x)+\int_{A\cap B} b_k(\x)P_n(d\x)=\\
&\frac{1}{Z_n}\int_B \int_{A_k(\x)}{w(x_1)\cdots w(x_n)w(x_{n+1})\over
w(x_{n+1})+w(x_k)}P^{n+1}_{\xi}(d\x,dx_{n+1})+\int_{A\cap B} b_k(\x)P_n(d\x)=(\$).
\end{align*}
Now, if we define
\begin{align*}
BA_k&:=\{\x \in \mathcal{X}^{n+1}\colon (x_1, \ldots, x_n) \in B, (x_1, \ldots,x_{k-1}, x_{n+1},x_{k+1},\ldots, x_n) \in A\}\\
AB_k&:=\{\x \in \mathcal{X}^{n+1}\colon (x_1, \ldots, x_n) \in A, (x_1, \ldots,x_{k-1}, x_{n+1},x_{k+1},\ldots, x_n) \in B\},
\end{align*}
we have that one set can be obtained from the other by swapping $x_k$ and $x_{n+1}$. Whence
\begin{align*}
&\int_B \int_{A_k(\x)}\frac{w(x_1)\cdots w(x_{n+1})}{
w(x_{n+1})+w(x_k)}P^{n+1}_{\xi}(d\x,dx_{n+1})=\\
&\int_{BA_k} \frac{w(x_1)\cdots w(x_{n+1})}{
w(x_{n+1})+w(x_k)}P^{n+1}_{\xi}(dx_1, \ldots,dx_k, \ldots,dx_{n+1})=\\
&\int_{BA_k} \frac{w(x_1)\cdots w(x_k) \cdots w(x_{n+1})}{
w(x_{n+1})+w(x_k)}P^{n+1}_{\xi}(dx_1, \ldots,dx_{n+1}, \ldots dx_k)=\\
&\int_{AB_k} \frac{w(x_1)\cdots w(x_{n+1}) \cdots w(x_{k})}{
w(x_{k})+w(x_{n+1})}P^{n+1}_{\xi}(dx_1, \ldots,dx_{k}, \ldots dx_{n+1})=\\
&\int_A\int_{B_k(\x)}{w(x_1)\cdots w(x_{n+1})\over
w(x_{n+1})+w(x_k)}P^{n+1}_{\xi}(d\x,dx_{n+1}).
\end{align*}
Although the sets $BA_k$ and $BA_k$ are, in general, different, the last equality holds because the function as well as the measure is invariant with respect to change $x_{n+1}$ and $x_k$.
Thus,
\[
\begin{split}
 (\$)&=
\frac{1}{Z_n}\int_A \int_{B_k(\x)}{w(x_1)\cdots w(x_{n+1})\over
w(x_{n+1})+w(x_k)}P^{n+1}_{\xi}(d\x,dx_{n+1})+\int_{A\cap B} b_k(\x)P_n(d\x)=\int_A P_k(\x,B
)P_n(d\x)\\
\end{split}
\]
and this concludes the first part of the proof.\\\\
For the inverse fitness kernel, we proceed similarly. Clearly,
$$\int_B \widetilde P(\x,A)P_n(d\x)=\sum_{k=1}^n\int_B\widetilde P_k(\x,A)P_n(d\x)+\int_{B\cap A} c(\x)P_n(d\x).$$
Let now $k=1,\ldots,n$  be fixed and, as previously, we obtain
\begin{align*}
&\int_B \widetilde P_k(\x,A)P_n(d\x)=\\
&\frac{1}{Z_n} \int_B
\int_{ A_k(\x)}
 \frac{w(x_1) \cdots w(x_n) }{\sum_{j=1}^{n+1} w(x_k)/w(x_j)}
P_\xi(dx_{n+1}| \x)P^n_\xi(d\x)=\\
&\frac{1}{Z_n} \int_{BA_k}
\frac{w(x_k)w(x_{n+1})}{\prod_{j=1}^{n+1} w(x_j)}\Big({\sum_{j=1}^{n+1} {1\over w(x_j)}}\Big)^{-1}
P^{n+1}_\xi(dx_1, \ldots , dx_{n+1})=\\
&\frac{1}{Z_n} \int_{BA_k}
\frac{w(x_k)w(x_{n+1})}{\prod_{j=1}^{n+1} w(x_j)}\Big({\sum_{j=1}^{n+1} {1\over w(x_j)}}\Big)^{-1}
P^{n+1}_\xi(dx_1, \ldots, dx_{n+1})=\\
& \int_A \widetilde P_k(\x,B)P_n(d\x).
\end{align*}
\end{proof}
\subsection{The measures $P_n$ and $Q_n$}\label{sec:PandQ} In the previous section we saw that
the measure $P_n$ defined as in \eqref{pn} is a stationary measure
for different selection schemes. The main objective of the current
article is to study the asymptotic behavior of $P_n$ as the
population size $n$ increases. To be more general, we shall assume
that the  fitness functions $w_n$ and the prior measures $\pi_n$
depend on $n$ hence, for every $n$, the measure $P_n$ on $\B(\X^n)$
is the following
\begin{equation}\label{pn1}
P_n(A)={1\over Z_n}\int_A \prod_{j=1}^n
w_n(x_j)P^n_{\xi}(d\x)={1\over Z_n}\int_{\P} \int_A \prod_{j=1}^n
w_n(x_j)q(dx_j)\pi_n(dq).
\end{equation}
The second equality in \eqref{pn1} follows from the fact that,
for every nonnegative measurable $f: \X^n\to \mathbb{R}^+$, it holds
$$\int_{\X^n} f(\x)P_{\xi}^n(d\x)=\int_{\X^n}\int_{\P} f(\x)q(dx_1)\cdots q(dx_n)\pi_n(dq).$$
Indeed, if $f=\ident_A$, then
$$P_{\xi}^n(A)=\int_{\P}q^n(A)\pi_n(dq)=\int_{\P}\int_{\X^n} f(\x)q(dx_1)\cdots q(dx_n)\pi_n(dq).$$
By linearity, the same holds for simple functions and then extends to nonnegative measurable functions by using the monotone convergence theorem and Fubini-Tonelli's theorem.
\\\\
Thus, if $A=A_1\times \cdots \times A_n$, then
\begin{equation}\label{pndef}
P_n(A)={1\over Z_n}\int_{\P} \Big(\prod_{j=1}^n \int_{A_j}
w_n(x)q(dx)\Big)\pi_n(dq),\end{equation} and now it is easy to see
that
\begin{equation}\label{zn}
Z_n=\int_\P \langle w_n,q\rangle^n \pi_n(dq).\end{equation}
\paragraph{An alternative representation of $P_n$.} It turns out that it is convenient to represent the measure $P_n$
slightly differently as follows. For every $q\in \P$, we define a
probability measure $r_{q,n}$ on $\B(\X)$:
\begin{equation}\label{eq:rqn}
r_{q,n}(A):={\int_A w_n(x)q(dx)\over \langle w_n,q \rangle},\quad A \in \B(\X).
\end{equation}
The mapping $q\mapsto \langle w_n,q \rangle$ is  continuous,  and it
can also be shown that for any fixed $A\in \B(\X)$ the mapping
$q\mapsto \int_A w_n(x)q(dx)$ is measurable.
Indeed, $q \mapsto q(A)$ is measurable for all $A \in \B(\X)$, hence by linearity $q \mapsto \int_\X f(x)q(dx)$ is measurable for all measurable simple function $f$; for a generic
nonnegative measurable function $f$, the result follows by taking the usual limit argument $f_n\uparrow f$ where $\{f_n\}_{n}$ are simple functions.
And so for
any $A$, the mapping $q\mapsto r_{q,n}(A)$ is measurable as well. By
$\pi-\lambda$ argument, for any $A\in \B(\X^n)$, $q\mapsto
r^n_{q,n}(A)$ is measurable, where $r^n_{q,n}$ stands for $n$-fold
product measure.\\\\
Given a probability measure $\pi_n$ on ${\cal B}(\P)$, we define
another probability measure $\bp$ on ${\cal B}(\P)$ as follows
\begin{equation}\label{eq:barpin}
\bp(E):={1\over Z_n}\int_E \langle w_n,q \rangle^n \pi_n(dq),\quad E\in {\cal B}(\P) .
\end{equation}
Here $Z_n$ is the normalizing constant, thus $Z_n$ is as in \eqref{zn}.
Now, the measure $P_n$ can be alternatively defined as follows
\begin{equation}\label{pn2}
P_n(A)=\int_{\P}r^n_{q,n}(A)\bp(dq),\quad A\in \B(\X^n).
\end{equation}
To see that the equality \eqref{pn2} holds, observe that the right
hand side of \eqref{pn2} defines a probability measure that  for any
measurable cylinder $A=A_1\times \cdots \times A_n$ reads
\begin{equation*}
\int_{\P}r^n_{q,n}(A)\bp(dq)=\int_\P \prod_{i=1}^n
r_{q,n}(A_i)\bp(dq)={1\over Z_n}\int_\P
 \Big(\prod_{i=1}^n \int_{A_i}w_n(x)q(dx) \Big)\pi_n(dq)=P_n(A),\end{equation*}
where the last equality holds by \eqref{pndef}. Therefore these
measures coincide on cylinders, hence also on $\B(\X^n)$.\\\\
We can go one step further, and consider the mapping
\begin{equation}\label{rn-map}
r_n: \P \mapsto \P,\quad r_n(q):=r_{q,n}.\end{equation} For every
$n$, the map $r_n$ is continuous: let
$f$ be a bounded continuous function on $\X$. Note that $w_n$ and $f\cdot w_n$ and bounded and continuous for every $n$.
Whence,
\begin{multline*}
 \int_\X f(x) r_{q_m,n}(dx)=\frac{1}{\langle w_n, q_m \rangle } \int_\X f(x) w_n(x) q_m(dx)
 \\
 \stackrel{m \to \infty}{\longrightarrow}
 \frac{1}{\langle w_n, q \rangle } \int_\X f(x) w_n(x) q(dx)= \int_\X f(x) r_{q,n}(dx).
 \end{multline*}
Thus $r_n(q_m)=r_{q_m,n}\Rightarrow r_{q,n}=r_n(q)$, whence
$r_n$ is  measurable. Now
define the pushforward measure $\nu_n$ on $\B(\P)$ as follows
$\nu_n(E):=\bp\big(r_n^{-1}(E)\big)$. Thus by change of variable
formula
\begin{equation}\label{nun-def}
P_n(A)=\int_{\P}r^n_{q,n}(A)\bp(dq)=\int_{\P}q^n(A)\nu_n(dq),\quad A\in \B(\X^n).\end{equation}
\paragraph{The measure $Q_n$.} Recall the mapping $g$ defined in \eqref{defg} and the measure $Q_n$ defined in \eqref{eq:Qn}.
The mapping $g$  is many-to-one, because all permutation of a vector
$\x$ have the same $g(\x)$. Observe that for any function $f$ on
$\X^n$, it holds: $n^{-1}\sum_{i=1}^n f(x_i)=\langle f,
g(\x)\rangle$. This observation helps us to see that the mapping $g$
is continuous. Let $\x^m\to \x$ be a convergent sequence in $\X^n$.
Since the convergence in $\X^n$ is equivalent to pointwise
convergence, it follows that as $m\to \infty$,
for any continuous and bounded function on $\X^n$, i.e for any $f\in
C_b(\X^n)$, it holds
$$\langle f,g(\x^m)\rangle = n^{-1}\sum_{i=1}^n f(x^m_i)\to  n^{-1}\sum_{i=1}^n f(x_i)=\langle f,
g(\x)\rangle.$$ So the convergence $\x^m\to \x$ implies
$g(\x^m)\Rightarrow g(\x)$, hence $g$ is continuous and  measurable.
 The advantage of $Q_n$ over $P_n$
is that, for every $n$, the measure $Q_n$ is defined on the same
domain $\B(\P)$, and so one can study the convergence on $Q_n$ in
the usual sense of weak convergence of probability measures. When
$\X$ is finite, then the measure $Q_n$ can be constructed
explicitly, see \cite{article1}.
\section{Limit process and limit measure $Q^*$}\label{sec:limit}
\subsection{The limit process}
We now turn to the asymptotics of $P_n$ as $n$ grows. Recall that we aim to show the existence of a limit process $X_1,X_2,\ldots$ so that
\eqref{fi-con} holds, where $(X_{1,n},\ldots,X_{n,n})\sim P_n$.
Observe that \eqref{fi-con} is equivalent to the following: for any $m\in
\mathbb{N}$,
\begin{equation}\label{fi-con2}
(X_{1,n},\ldots,X_{m,n})\Rightarrow (X_1,\ldots,X_m).
\end{equation}
Indeed, from \eqref{fi-con2}, it follows that \eqref{fi-con} holds
when $t_1<t_2<\ldots<t_m$ and the weak convergence of random vectors
implies that of the permutations. According to \eqref{pn2}, for
every $A_i\in \B(\X)$, $i=1,\ldots,m$, it holds
$${\bf P}(X_{1,n}\in A_1,\ldots,X_{m,n}\in A_m)=\int_\P \prod_{i=1}^m
r_{q,n}(A_i)\bp(dq)=:P_n(A_1\times \cdots \times A_m).$$
By the canonical representation, the existence of a stochastic process
is equivalent to the existence of a probability measure $P^*$ on
$(\X^{\infty},\Sigma)$, where $\Sigma$ is the product
$\sigma$-algebra. The measure $P^*$ can be considered as the
distribution of $X$. Let $\mathcal{C}(A_1\times \cdots \times A_m):=\{(x_i)\in
\X^{\infty}: x_1\in A_1,\ldots,x_m\in A_m\}$ be any
measurable cylinder, where $A_i\in \B(\X)$, $i=1,\ldots,m$. With
slight abuse of notation, we shall denote by $P^*(A_1\times \cdots
\times A_m)$ the measure of the cylinder $\mathcal{C}(A_1\times \cdots \times A_m)$. If $P^*$ is the
distribution of $X$, then $P^*(A_1\times \cdots \times A_m)={\bf
P}(X_1\in A_1,\ldots, X_m\in A_m)$.
 Since cylinders are a
convergence-determining class, (\cite{Billingsley}, Theorem 2.8), the
convergence \eqref{fi-con2} holds if for every $m$, and every
measurable and $P^*$-continuous cylinder $A=\{(x_i)\in \X^{\infty}:
x_1\in A_1,\ldots,x_m\in A_m\}$ it holds
\begin{equation}\label{two}
P_n(A_1\times \cdots \times A_m)\to P^*(A_1\times \cdots \times
A_m).\end{equation}
Recall that  $A$ is  $P^*$ continuous when $P^*(\partial A)=0$,
where $\partial A$ stands for the boundary of $A$.
To summarize, for showing \eqref{fi-con}, it suffices to show
the existence of a probability measure $P^*$ on $(\X^{\infty},\Sigma)$
such that for all  $P^*$-continuous cylinders \eqref{two} holds.\\\\
In the following theorem $\{r_{q,n}\}_n$
are
probability measures on $\X$ which do not necessarily coincide with the ones defined by \eqref{eq:rqn} (with $w_n$ as in
\eqref{wn}). We will see that for that particular choice of measures $\{r_{q,n}\}_n$, by Corollary \ref{cor1}, the hypotheses in Theorem
\ref{thmP} concerning uniform convergence to $r_q$ and  continuity of the map $q\to r_q$ are always satisfied.
\begin{theorem}\label{thmP} Suppose there exists a probability measure $\bar{\pi}$ on  ${\cal P}$ such that $\bp\Rightarrow \bar{\pi}$. For every $q \in \P$ and $n \in \mathbb{N}$, let
$r_{q,n},\, r_q \in \P$ be such that
 $\sup_{q\in {\cal P}}|r_{q,n}(A)-r_q(A)|\to
0$ for
 all $A\in {\cal B}(\X)$. Assume also that
$q \mapsto r_q$ is continuous. Then there
exists an infinitely exchangeable  process $X$ so that for every
$m \in \mathbb{N}$,  the convergence \eqref{fi-con2} holds. Moreover, the limit process is such that for every $m\in \mathbb{N}$ and $A_1,\ldots,A_m\in \B(\X)$,
\begin{equation}\label{eq:p*}
P^*\big(A_1\times\cdots \times A_m\big):=
\int_{\P} r^m_{q}\big(A_1\times\cdots \times A_m\big)\p(dq).
\end{equation}
\end{theorem}
\begin{proof}
Let for every $m$  and distinct integers  $t_1,\ldots,t_m\in \mathbb{N}$,
$\mu_{t_1,\ldots,t_m}$ be a probability measure on $\B(\X^m)$
defined as follows: $$ \mu_{t_1,\ldots,t_m}(A):=\int_{\P}
r^m_{q}(A)\bar{\pi}(dq),\quad A\in \B(\X^m).$$ The definition is
correct, because by assumption $q\mapsto r_q$ is measurable, and so
for every $m$ and $A\in \B(\X^m)$, the mapping $q\mapsto r^m_q(A)$
(product measure) is measurable as well. Note that this definition depends on $m$ but is independent of the choice of $t_1, \ldots, t_m$. Clearly the family
$\{\mu_{t_1,\ldots,t_m}\}$ fulfills the consistency conditions, and
so by Kolmogorov existence theorem there exists a measure $P^*$ on
$(\X^{\infty},\Sigma)$ such that for every distinct integers
$t_1,\ldots, t_m$ and every $A\in \B(\X^m)$, it holds
$$P^*\big(\{ (x_i)\in \X^{\infty}: (x_{t_1},\ldots,x_{t_m})\in
A\}\big)=\mu_{t_1,\ldots,t_m}(A).$$
In particular
$$P^*(A_1\times \cdots \times A_m)=\mu_{1,\ldots,n}(A_1\times \cdots \times
A_m)=\int_{\P}\prod_{i=1}^mr_q(A_i){\bar \pi}(dq).$$  Thus $P^*$ is
the distribution of an  infinitely exchangeable process, and the
theorem is proven, when we show that \eqref{two} holds for all
$P^*$-continuous cylinders.
To show \eqref{two}, we use Skorohod representation theorem (\cite{
Billingsley}, Theorem 6.7) according to which there are ${\cal
P}$-valued random variables $Z_n$ and $Z$ such that $Z_n$ has
distribution $\bp$, $Z$ has distribution ${\bar \pi}$ and $Z_n\to Z$
a.s.. The theorem applies on separable metric space, but ${\cal P}$
equipped with Prokhorov metric is separable.
 Fix a $P^*$-continuous cylinder $A=\{(x_i)\in \X^{\infty}: x_1\in A_1,\ldots,x_m\in A_m\}$
 and let us denote
$$f_n(q):=\prod_{i=1}^m r_{q,n}(A_i)=r^m_{q,n}(A_1\times \cdots \times A_m),\quad f(q):=\prod_{i=1}^m r_{q}(A_i)=r^m_{q}(A_1\times \cdots \times A_m).$$
If $m=1$, then by assumption, it immediately follows that
$\sup_q|f_n(q)-f(q)|\to 0$. Since the functions are bounded, the
uniform convergence also holds when $m>1$. Indeed, if $f_n\to f$ and
$g_n\to g$ uniformly and all functions are bounded by 1, then
$$|f_ng_n-fg|=|f_ng_n-f_ng+f_ng-fg|\leq|f_n(g_n-g)|+|g(f_n-f)|\leq
|f_n-f|+|g_n-g|.$$ Let $E_{\rm cont}\subset {\cal P}$ be the set of
continuity points of $f$. We shall show that ${\bar \pi}(E_{\rm
cont})=1$. Since
$$\partial A=\{(x_i)\in \X^{\infty}: (x_1,\ldots,x_m)\in \partial (A_1\times \cdots \times A_m) \},$$
 where $\partial (A_1\times \cdots \times A_m)$ is
the boundary in $\X^m$, we have (since $A$ is $P^*$-continuous)
$$P^*(\partial A)=\int_\P r^m_q\big(\partial (A_1\times \cdots \times A_m)\big){\bar \pi}(dq)=0.$$
The integral of non-negative function is zero only if the function
is ${\bar \pi}$-a.s. equal to 0 and so we have ${\bar \pi}(F)=1$,
where
$$F:=\{q: r^m_q\big(\partial (A_1\times \cdots \times A_m)\big)=0\}.$$
We now show that $F\subset E_{\rm cont}$. Indeed, if $q\in F$ and
$q_n\Rightarrow q$, then by the continuity assumption
$r_{q_n}\Rightarrow r_q$ and thus also $r^m_{q_n}\Rightarrow r^m_q$
(because weak convergence of marginal measures implies weak
convergence of the product measure). Since $r^m_q\big(\partial
(A_1\times \cdots \times A_m)\big)=0$, it follows that
$$r_{q,n}^m(A_1\times \cdots \times A_m)\to r^m_q(A_1\times \cdots \times
A_m).$$ Thus $f(q_n)\to f(q)$ and so $q\in E_{\rm cont}$. Since
${\bar \pi}(E_{\rm cont})=1$, from $Z_n\to Z$ a.s. it follows
$f(Z_n)\to f(Z)$ a.s.
%
%
From the uniform convergence, it follows that
$|f_n(Z_n)-f(Z_n)|\to 0$. These two facts together imply
$f_n(Z_n)\to f(Z)$, a.s. Finally, since the functions $f_n$ are all
bounded by 1, by the bounded convergence theorem it follows
$Ef_n(Z_n)\to Ef(Z)$. Since
\begin{align*}
&P_n(A_1\times \cdots \times A_m)=\int_\P r^m_{q,n}(A_1\times \cdots
\times A_m)\bp(dq)=\int_\P
f_n(q)\bp(dq)=Ef_n(Z_n)\\
&P^*(A_1\times \cdots \times A_m)=\int_\P r^m_{q}(A_1\times \cdots
\times A_m){\bar \pi}(dq)= \int_\P f(q){\bar
\pi}(dq)=Ef(Z),\end{align*} we have \eqref{two}.\end{proof}
\begin{corollary}\label{cor1}
 Let $r_{q,n}$ and $w_n$ be as defined in \eqref{eq:rqn} and \eqref{wn} respectively.
If $\lambda>0$, define  $r_q=q$, while if $\lambda=0$, let $r_q$ be the measure proportional to $wdq$, where $w=w_n$ (in this case $w_n$ does not depend on $n$).
Then
\begin{enumerate}
\item $\sup_{q\in {\cal P}}|r_{q,n}(A)-r_q(A)|\to0$ for all $A\in \B(\X)$;
\item $q\mapsto r_q$ is continuous;
\item if $\bp\Rightarrow \bar{\pi}$ (where  $\bp$ is defined in \eqref{eq:barpin}) then $P_n$ converges (in the sense of \eqref{fi-con2})
to the measure $P^*$
defined in \eqref{eq:p*}.
\end{enumerate}
\end{corollary}
\begin{proof}
\begin{enumerate}
\item Recall that $q\mapsto r_{q,n}$ is continuous as explained after \eqref{rn-map} and
 that $w_n(x)=\exp(-\phi(x)/n^\lambda)$. Let $w=1$ if $\lambda>0$, $w=\exp(-\phi(x))=w_n(x)$ if $\lambda=0$.
\\
Note that $\sup_{x\in\X}|w_n(x)-w(x)|=0$ if $\lambda=0$.
If $\lambda>0$, then $\sup_{x\in\X}|w_n(x)-w(x)|=1-\exp(-\sup_x\phi(x)/n^\lambda)$, which goes to 0 as $n$ tends to infinity, since
$\phi$ is by hypothesis bounded.
It  follows that
$$\sup_q |\langle w_n,q \rangle - \langle w,q \rangle| \le \sup_q \langle |w_n-w|,q \rangle
\to 0.$$
Clearly the same argument holds when integrating over any set $A$.
We now observe that $\inf_{x\in \X} w(x)=\exp(-\sup_x\phi(x)/n^\lambda)>0$, when $\lambda>0$ and it is equal to 1 when $\lambda=0$.
This implies that,  for all $A\in \B(\X)$,
$$\sup_{q\in {\cal P}}|r_{q,n}(A)-r_q(A)|=\sup_{q\in {\cal P}}\Big|{\langle \ident_A w_n,q \rangle \over \langle w_n,q \rangle}-{\langle \ident_A w,q \rangle \over \langle w,q \rangle}\Big|\to 0.$$

\item
Although the continuity of $q\mapsto r_{q}$ follows by the same argument as the continuity of $q\mapsto r_{q,n}$, we shall now show that it can be directly deduced from the continuity of
 $q\mapsto r_{q,n}$ and the uniform convergence stated in 1.
Let $f$ be a bounded, nonnegative, measurable function on $\X$ and consider a sequence $\{q_m\}_m$ such that $q_m \Rightarrow q$.
  By using the Uniform Bounded Convergence Theorem (see for instance \cite[Theorem 2.3]{cf:BZ02}), as $n\to \infty$,
 \[
  \sup_{q \in \P} \Big | \int_\X f(x) r_{q,n}(dx) -
  \int_\X f(x) r_{q}(dx) \Big |
 =
  \sup_{q \in \P} \Big | \int_0^{+\infty}  r_{q,n}(f \ge t) dt -
  \int_0^{+\infty} r_{q}(f \ge t) dt \Big | \to 0.
\]
Suppose now that, in addition, $f$ is continuous. Take $\varepsilon >0$ and $n_0=n_0(\varepsilon)$ such that for all $n \ge n_0$ we have
\[
  \sup_{q \in \P} \Big | \int_\X f(x) r_{q,n}(dx) -
  \int_\X f(x) r_{q}(dx) \Big | <\varepsilon/3.
\]
Since $q \mapsto r_{q,n}$ is continuous for every $n \in \mathbb{N}$, take  $m_0=m_0(\varepsilon, n_0)$ such that
\[
 \Big | \int_\X f(x) r_{q_m,n_0}(dx) -
  \int_\X f(x) r_{q,n_0}(dx) \Big | < \varepsilon/3
\]
for all $m \ge m_0$. Clearly, for all $m \ge m_0$
\[
 \begin{split}
  \Big | \int_\X f(x) r_{q_m}(dx) -
  \int_\X f(x) r_{q}(dx) \Big |
  & \le
  \Big | \int_\X f(x) r_{q_m}(dx) -
  \int_\X f(x) r_{q_m, n_0}(dx) \Big |\\
  &+
  \Big | \int_\X f(x) r_{q_m,n_0}(dx) -
  \int_\X f(x) r_{q,n_0}(dx) \Big |\\
  &+
  \Big | \int_\X f(x) r_{q,n_0}(dx) -
  \int_\X f(x) r_{q}(dx) \Big | < \varepsilon.
 \end{split}
 \]
\item If $\bp\Rightarrow\bar \pi$ then all the assumptions of Theorem \ref{thmP} are satisfied and the claim follows.
\end{enumerate}
\end{proof}
\begin{remark}
The proof of Corollary \ref{cor1} shows that the assumptions of Theorem \ref{thmP} on $r_{q,n}$ and $r_q$ are satisfied also with more general weight functions than those considered in \eqref{wn}.
Indeed
\begin{itemize}
\item
if the weight functions are such that
$q\mapsto r_{q,n}$ is continuous for every $n \in \mathbb{N}$ and
 $\sup_{q\in {\cal P}}|r_{q,n}(A)-r_q(A)|\to
0$ for all $A\in {\cal B}(\X)$, then $q \mapsto r_q$ is continuous;
\item
in particular,  if
$w_n$ and $w$ are measurable functions so that
$\inf_{x \in \X} w(x)>0$ and $\sup_{x \in \X} |w_n(x)-w(x)| \to 0$,
then
$q\mapsto r_{q,n}$ is continuous for every $n \in \mathbb{N}$ and
 $\sup_{q\in {\cal P}}|r_{q,n}(A)-r_q(A)|\to
0$ for all $A\in {\cal B}(\X)$ with $r_q$ being the probability measure
proportional to $ w dq$.
\end{itemize}
\end{remark}

\subsection{The weak convergence of $Q_n$}
The goal is to show that under the same assumptions as in Theorem
\ref{thmP} with the additional request that
$\X$ is compact, the measures
$Q_n$ converge to a measure $Q$.
Recall the function $g$ in \eqref{defg} that maps  every sequence
${\bf x}=(x_1,\ldots,x_n)$ to its empirical measure. To stress out the
dependence of $n$, it this section, we shall denote the function $g$ as
$g_n$.
\begin{theorem}\label{thmQ} Suppose $\X$ is compact and the
assumptions of Theorem \ref{thmP} hold. Let $Q=\p r^{-1}$,
where $r:\P\mapsto\P$ is defined from the $r_q$ in Theorem \ref{thmP}, as $r(q)=r_q$ for all $q\in\P$.
Then $Q_n\Rightarrow Q$.
\end{theorem}
\begin{proof} By the Portmanteu theorem (see for instance
\cite[Theorem 2.1]{Billingsley}), it suffices to show that for every open set $E \in \B(\P)$ we have
$\liminf_n Q_n(E)\geq Q(E)$. Recall that, according to \eqref{nun-def},
$$Q_n(E)=\int_{\P}r^n_{q,n}\big(g_n^{-1}(E)\big)\bp(dq)=\int_{\P}q^n(g_n^{-1}(E)\big)\nu_n(dq),\quad E\in \B(\P).$$
Let $E$ be an open set. We first show that $\lim\inf_n \nu_n(E)\geq Q(E)$.
For that, we use Skorohod representation again, so let $Z_n\sim \bp$, $Z\sim
\p$ be $\P$-valued random variables so that $Z_n\to Z$, a.s.. Recall  that we use parallel notation $r_n(q):=:r_{q,n}$
and $r(q):=:r_q$, where $r_n,r: \P\to \P$. We now argue that the a.s. convergence $Z_n\to Z$ (with respect to Prokhorov metric) entails
 $r_n(Z_n)\to r(Z)$, a.s.. Since the convergence with respect to Prokhorov metric is equivalent to to the weak convergence of measures, it suffices to show that
 $q_n\Rightarrow q$ implies $r_n(q_n)\Rightarrow r(q)$. To see that take
 $A$ to be a $r_q$-continuous set. Then
$$|r_{q_n,n}(A)-r_q(A)|\leq
|r_{q_n,n}(A)-r_{q_n}(A)|+|r_{q_n}(A)-r_q(A)|.$$ By assumption,
$|r_{q_n,n}(A)-r_{q_n}(A)|\leq \sup_{q}|r_{q,n}(A)-r_{q}(A)|\to 0$,
 and since $q\mapsto r(q)$ is continuous, it follows that
$|r_{q_n}(A)-r_q(A)|\to 0$. Hence $r_n(q_n)\Rightarrow r(q)$, and so $r_n(Z_n)\to r(Z)$. Since $E$ is open, it follows that
$$P \big (\{r(Z)\in E\}\setminus
\liminf_n
\{r_n(Z_n)\in E\} \big )=0$$ and so the following holds for all open sets $E$
\begin{equation}\label{eq:Qmun}
\begin{split}
Q(E)&=\p\big(r^{-1}(E)\big)=P( r(Z)\in E)\leq P\big(\liminf_n
\{r_n(Z_n)\in E\}\big)\\
&\leq \liminf_n P(r_n(Z_n)\in E)=\liminf_n
\bp\big(r_n^{-1}(E)\big)=\liminf_n \nu_n(E).
\end{split}
\end{equation}
Denote by $m(\delta)$ the $\delta$-covering number (i.e.~the minimal number of $\delta$-balls needed to
cover $\P$), $E_{\delta}=\{p \in \P: d(p,E)\leq
\delta\}$ the closed $\delta$-blowup of $E$
and recall the definition of relative entropy
\begin{equation}\label{eq:entropy}
 D(p \| q):=
 \begin{cases}
  \int_\X \ln \big (\frac{dp}{dq} \big ) dp & \text{if }p \ll q\\
  +\infty & \text{else}.
 \end{cases}
\end{equation}
When $\X$ is compact,
then also $\P$ is compact, so $m(\delta)<\infty$. For us it is
important that $m(\delta)$ is independent of $q$.\\
Since $\X$ is compact, for every $q \in \P$ the following inequality holds (see \cite{LDP}, Ex 6.2.19):
\begin{equation}\label{sanov2}
q^n\big(g_n^{-1}(E)\big)\leq
\inf_{\delta>0}\Big(m(\delta)\exp[-\inf_{p\in
E_{\delta}}D(p\|q)\cdot n]\Big).
\end{equation}
Let $E^c_{\delta}$ be closed
$\delta$-blowup of $E^c$. Then, for any $\delta>0$,
define
$$F_{\delta}:=\big(E^c_{2\delta}\big)^c.$$
 Clearly $F_{\delta}$ is an open
set inside $E$, and $\cup_{\delta>0}F_{\delta}=E$. We now argue that
for any $\delta>0$,
\begin{equation}\label{inside}
\inf_{p \in E^c_\delta, q \in F_\delta} D(p\|q)>0.
\end{equation} If not, there would
be a sequence $\{q_n\}_n$ in $F_{\delta}$ and $\{p_n\}_n$ in
$E^c_{\delta}$
so that $D(p_n\|q_n)\to 0$. From Pinsker's inequality, it follows
that $d(p_n,q_n)\to 0$, where $d$ stands for Prokhorov metric. But
since $p_n\in E^c_{\delta}$ and $q_n\in \big(E^c_{2\delta})^c$, it
must be that $d(p_n,q_n)>\delta$ for every $n$. Hence \eqref{inside} holds. From \eqref{sanov2},
we obtain that for every $\delta>0$,
\begin{equation}\label{lb}
\sup_{q\in F_{\delta}}q^n\big(g^{-1}_n(E^c)\big)\leq
m(\delta)\exp[-\varepsilon(\delta)\cdot n],\quad \inf_{q\in
F_{\delta}}q^n\big(g^{-1}_n(E)\big)\geq
1-m(\delta)\exp[-\e(\delta)\cdot n].\end{equation} Therefore,
\begin{equation*}\label{alt}
Q_n(E)\geq \int_{F_{\delta}}q^n(g^{-1}_n(E))\nu_n(dq)\geq
\big(1-m(\delta)\exp[-\e(\delta)\cdot n] \big)
\nu_n(F_{\delta}).
\end{equation*}
From equation~\eqref{eq:Qmun}, since $F_{\delta}$ is open  we deduce that
$$\liminf_n Q_n(E)\geq \liminf_n \nu_n(F_{\delta})\geq
Q(F_{\delta}) \uparrow Q(E), \quad \textrm{as } \delta \downarrow 0$$
where the last limit follows from the continuity of a measure.
\end{proof}
\begin{remark}\label{rem:1}
 The additional assumption of compactness is disappointing. In the proof above, it was needed for the Sanov type of inequality \eqref{sanov2}, which, in turn, was needed for the uniform convergence in \eqref{lb}. The convergence \eqref{lb} is, in a sense, exponentially fast, but for our proof the speed is not important, it just suffices to have:
\begin{equation}\label{unif}\inf_{q\in F_{\delta}}q^n(g_n^{-1}(E))\to 1.\end{equation}
Recall that for every open $E$, and for every $q\in E$, by SLLN $q^n(g^{-1}_n(E))\to 1$. The convergence \eqref{unif} states that on the set $F_{\delta}$ this convergence is uniform and then our proof applies.
\end{remark}
\begin{remark}\label{rem:2}
In \cite{article1}, it was shown that, when $\X$ is finite,
for any continuous and bounded $f: \P\to \mathbb{R}$, it holds: $\int_\P f dQ_n\to \int_\P f dQ$. In our notation
$$\int_\P f(q)Q_n(dq)=\int_{\P}\int_{\X^n}f(g_n(\x))q^n(d\x)\nu_n(d q).$$
By SLLN,
$$f_n(q):=\int_{\X^n}f(g_n(\x))q^n(d\x)\to f(q),$$
and if this convergence were uniform, i.e.
 \begin{equation}\label{fuu}
\sup_q|f_n(q)-f(q)|\to 0,
\end{equation} then from $\nu_n\Rightarrow Q$, by using the Skorohod representation, it would follow that $\int_\P f_n(q)\nu_n(dq)\to \int_\P f(q)Q(dq)$. In \cite{article1} the equation \eqref{fuu} for finite $\X$ was obtained with Bernstein polynomials.
\end{remark}
\section{Arbitrary prior $\pi$}\label{sec:pi}
In this section, we consider the case when the prior $\pi$ is arbitrary and independent of $n$.
In what follows, we take $w_n$ as defined in \eqref{wn}. By Theorem \ref{thmP} and Corollary \ref{cor1},
to ensure the existence of the limit process,
it suffices to show $\bp\Rightarrow \p$.\\\\
Since in the integral \eqref{eq:barpin} defining $\bp$ we have the function $\langle w_n, q \rangle^n$, we start with the following observation.
\begin{proposition}\label{Pee} If $m\to \infty$, and $\phi$ is nonnegative and integrable with respect to $q$ (but not
necessarily bounded), then
\begin{equation}\label{ee}
\langle \exp[-{\phi\over m}], q \rangle^m\to \exp[-\langle \phi, q
\rangle].\end{equation} Moreover, when $\phi$ is bounded, then the
convergence is uniform over $q$.
\end{proposition}
\begin{proof}
Let us consider an i.i.d.~sequence $\{X_i\}_{i \in \mathbb{N}}$ of random variables with law $q$.
Clearly, by using the Law of Large Numbers and the Bounded Convergence Theorem, we have
\[
\begin{split}
 \Big ( \int_\X \exp \Big ( -\frac{\phi(x)}{m} \Big ) q(dx) \Big )^m&= \mathbb{E} \Big [
 \prod_{i=1}^m \exp \Big ( -\frac{\phi(X_i)}{m} \Big )
 \Big ]\\
 &=
 \mathbb{E} \Big [
 \exp \Big ( -\sum_{i=1}^m  \frac{\phi(X_i)}{m} \Big )
 \Big ] \to
 \exp \big ( \mathbb{E}[- \phi(X_1)] \big )
 \end{split}
\]
since $\exp ( -\sum_{i=1}^m  \frac{\phi(X_i)}{m} ) \to \exp  ( \mathbb{E}[- \phi(X_1)] )$ a.s.~as $m \to \infty$.
\\\\
Suppose now that $\phi$ is bounded, say $|\phi(x)|\le K$ for all $x \in \X$; it is enough to prove that $\mathbb{E}[|Y_n - m_q|] \to 0$ as $n \to +\infty$ uniformly with respect to $q$, where $Y_n:= \sum_{i=1}^n  \phi(X_i)/n$ and $m_q:=\mathbb{E}[\phi(X_1)]$.
For every $\varepsilon >0$, given any $q$ we have
\[
 \begin{split}
  \mathbb{E}[|Y_n - m_q|] &=
  \mathbb{E}[|Y_n - m_q|\ident_{\{|Y_n - m_q| \le \varepsilon/2\}}] +
  \mathbb{E}[|Y_n - m_q|\ident_{\{|Y_n - m_q| > \varepsilon/2\}}]\\
  &\le \frac{\varepsilon}{2} + 2KP \big (|Y_n - m_q| > \varepsilon/2 \big )\le \frac{\varepsilon}{2}+
  \frac{8K^3}{n(\varepsilon/2)^2} \le \varepsilon
 \end{split}
\]
if $n \ge \frac{8K^3}{(\varepsilon/2)^3}$ (and this does not depend on $q$).
%
%
%
\end{proof}
\subsection{The case $\lambda\geq 1$}
Let us begin with the case $\lambda\geq 1$.
We establish the convergence of the measure $\bar \pi_n$ which was defined in \eqref{eq:barpin}. The next theorem states that when $\lambda>1$, then the influence of fitness vanishes, and the limit process $X$ equals the breeding process $\xi$. When $\lambda=1$, then the limit process is another infinitely exchangeable process whose prior measure differs from the breeding one $\pi$, and depends on $\phi$ as well as on $\pi$.
\begin{theorem}\label{thm1}
Let the fitness function be as in \eqref{wn} where $\phi$ is  non-negative, measurable and bounded. Suppose $\pi_n=\pi$; and let $\lambda\ge1$. Then the following convergences hold:
\begin{enumerate}[1)]
  \item If $\lambda=1$, then $\bp \Rightarrow \p$,
and $P_n\to P^*$ in the sense of \eqref{fi-con2}, where
\[
\begin{split}
& \p(E)  :=\frac1Z{\int_E}\exp[-\langle \phi, q
\rangle]\pi(dq),\quad {\rm where}\quad Z:=\int_\P \exp[-\langle \phi, q
\rangle]\pi(dq),\\
& P^*(A_1\times \cdots \times A_m) =\int_{\P} \prod_{i=1}^m q(A_i)\p(dq),
\quad \forall m\in\mathbb N, A_i\in\B(\X).
\end{split}
\]
If, in addition, $\X$ is compact, then $Q_n\Rightarrow \p$.
  \item If $\lambda>1$, then $\bp \Rightarrow \pi$
and $P_n\to P_\xi$ in the sense of \eqref{fi-con2}.
  If, in addition, $\X$ is compact, then $Q_n\Rightarrow \pi$.
\end{enumerate}
\end{theorem}
\begin{proof}
Before explicitly dealing with the two cases, we note that by Corollary \ref{cor1},  we only need to establish the convergence of $\bp$.
Since $r_q=q$, from Theorem \ref{thmP} we will get the convergence of $P_n$ to the limiting process.
\begin{enumerate}[1)]
  \item
Since for any $q$ and any $n$, it
holds $\langle \exp[-{\phi\over n}], q \rangle^n\leq 1$, we obtain
from \eqref{ee} and the Bounded Convergence Theorem that, for any $E\in \B(\P)$,
\begin{equation}\label{Ekoo}
Z_n \bar \pi_n(E)=\int_E \langle \exp[-{\phi\over n}], q \rangle^n \pi (dq)\to \int_E
\exp[-\langle \phi, q \rangle]\pi(dq).\end{equation}
From \eqref{Ekoo}, it follows that
$\bp(E)\to {\bar\pi}(E),$
meaning that  $\bp\Rightarrow \p$. Since the assumptions of Theorem \ref{thmP} are fulfilled with $w(x)=1$, then for compact $\X$   Theorem \ref{thmQ} implies $Q_n\Rightarrow \p$.
\item When $\lambda>1$, then by Proposition~\ref{Pee}, eventually as $n \to \infty$
$$ 1 \leftarrow  \big ( 2\exp(-\langle \phi,q \rangle) \big )^{1/n^{\lambda-1}} \ge
\langle \exp[-{\phi\over n^{\lambda}}], q \rangle^n
\ge
\Big ( \frac{\exp(-\langle \phi,q \rangle)}{2} \Big )^{1/n^{\lambda-1}}
\to 1$$
and by dominated
convergence, again, for any measurable $E$
\begin{equation}\label{Ekoo2}
\int_E \langle \exp[-{\phi\over n}], q \rangle^n \pi (dq)\to
\pi(E).\end{equation} Therefore $\bp\Rightarrow \pi$.
The convergence $Q_n\Rightarrow \pi$ is a consequence of Theorem \ref{thmQ}.
\end{enumerate}
\end{proof}
\paragraph{Remark.} Observe that for the weak convergence of $\bp$, the boundedness of $\phi$ is not needed. However, it is needed for the uniform convergence of $w_n\to w$, hence for existence of the limit process (Theorem \ref{thmP}) and for the weak convergence of $Q_n$ (Theorem \ref{thmQ}). Theorem \ref{thm1} is a direct generalization of
Theorem 5.1~(2)~and~(3) in \cite{article1}, and no additional assumptions are imposed.

\subsection{Case $\lambda\in [0,1)$}
\paragraph{Preliminaries: densities and powers.}
Let $\pi$ be a finite measure (not necessarily  a probability measure) on $\B(\P)$. Let  ${\cal S}$ be the support of $\pi$. For a measurable function $f: {\cal P}\to \mathbb{R}$,
$$\|f\|_{\infty}:={\rm ess sup} (f):=\inf \{c: |f|\leq c \quad \pi-\text{a.e.}\}.$$
If $f$ is continuous then
$\|f\|_{\infty}=\sup_{q\in {\cal S}}|f(q)|.$
Also recall that for any $0<m<\infty$
$$\|f\|_m:=\Big(\int_\P |f(q)|^m \pi(dq)\Big)^{1\over m}.$$
If $f$ is essentially bounded, $m$ grows and $\pi$ is a probability measure, then $\|f\|_m\nearrow
\|f\|_{\infty}< \infty$.
Then it follows that $\|f\|_m\to \|f\|_{\infty}$ also when $\pi$ is a finite (but not necessarily a probability) measure.
\\\\
We now consider
measurable functions $f_n, f: \P\to  \mathbb{R}^+$ , such that $f_n$ are essentially bounded and $\|f_n- f\|_\infty \to 0$ uniformly (whence $f$ is essentially bounded). Let $m_n\to \infty$ be an increasing sequence. We define a sequence of probability measures
$\mu_n$ on $\B(\P)$, where
$$\mu_n(E):=\int_E h_n(q) \pi(dq),\quad   h_n:={f_n^{m_n}\over \int_\P f_n^{m_n} d\pi}=\Big({f_n\over
\|f_n\|_{m_n}}\Big)^{m_n}.$$
Due to the $\|\cdot\|_\infty$-convergence and essential boundedness of $f$, the functions $f_n$ are essentially bounded as well, thus (recall $\pi$ is finite)
 $\int_\P  f_n^{m_n} d\pi<\infty$ for every $n$. In what follows, let
\begin{equation}\label{S}
{\cal S}^*:=\{q\in {\cal S}: f(q)=\|f\|_{\infty}\},\quad {\cal S}_{\delta}^*:=\{q\in {\cal S}: f(q)> \|f\|_{\infty}-\delta\}.\end{equation}
 The
following proposition is a generalization of \cite[Proposition 5.1]{article1}.
\begin{proposition}\label{point2}
Let $\pi$ be a finite measure on $\P$.
Let, $f_n$ and  $f$ be nonegative measurable functions on $\P$. Assume that $f_n$ are essentially bounded and $\|f_n-f\|_\infty \to 0$ as $n \to \infty$. 
Then for every $\delta>0$, $\mu_n\big({\cal
S}_{\delta}^*\big)\to 1$. Moreover, if, for some $\mu\in\P$, $\mu \big (\bigcap_{\delta>0} \overline{\mathcal{S}_\delta^*} \setminus \mathcal{S}^* \big )=0$ 
(for instance if $f$ is upper semicontinuous) and
$\mu_n \Rightarrow \mu$ then $\mu(\mathcal{S}^*)=1$.
\end{proposition}
\begin{proof} Since $\|f_n\|_{\infty}<\infty$
and $\|f_n-f\|_\infty \to 0$ then $\|f\|_{\infty}<\infty$.
By assumption, $\pi$ is a finite measure. Since $f_n$ converges to
$f$ uniformly, it follows that $\|f_n-f\|_{\infty}\to 0$ and so $\|
f_n\|_{\infty}\to \| f\|_{\infty}.$ For every $m$,
$$|\|f_n\|_m-\|f\|_m |\leq \|f_n-f\|_m\leq \pi(\P)^{1\over m}\|f_n-f\|_{\infty}\to 0.$$
Since $\|f\|_{m_n}\to \|f\|_{\infty}$, we have
\begin{align*}
\big|\|f_n\|_{m_n}-\|f\|_{\infty}\big|&\leq \big|\|f_n\|_{m_n}-\|f\|_{m_n}\big|+\big|\|f\|_{m_n}-|f\|_{\infty}\big|\\
&\leq \|f_n-f\|_{m_n}+\big|\|f\|_{m_n}-|f\|_{\infty}\big|\\
&\leq  \pi(\P)^{1\over m_n}\|f_n-f\|_{\infty}+\big|\|f\|_{m_n}-\|f\|_{\infty}\big|\to 0.
\end{align*}
 Now fix $\delta>0$ and note that $${\cal S}\setminus{\cal S}^*_{\delta}=\{q: f(q)\leq
 \|f\|_{\infty}-\delta\}. $$
Define $\delta':=\delta /\|f\|_{\infty}$. Then
\begin{align*}
\esssup_{q} \Big ( \ident_{{\cal S}\setminus{\cal S}^*_{\delta}} \frac{f_n(q)}{ \|f_n\|_{m_n}} \Big )&=\esssup_{q} \Big (\ident_{{\cal S}\setminus{\cal S}^*_{\delta}} \frac{f(q)+(f_n(q)-f(q))}{\|f_n\|_{m_n}} \Big )\\
&=
\esssup_{q} \Big ( \ident_{{\cal S}\setminus{\cal S}^*_{\delta}}\frac{f(q)+(f_n(q)-f(q))}{\|f\|_{\infty}}\frac{\|f\|_{\infty}}{\|f_n\|_{m_n}} \Big )\\
&\leq \esssup_{q} \Big ({\ident_{{\cal S}\setminus{\cal S}^*_{\delta}}} \frac{f(q)}{
\|f\|_{\infty}}\frac{\|f\|_{\infty}}{
\|f_n\|_{m_n}} \Big )+\frac{\|f_n-f\|_{\infty}}{\|f_n\|_{m_n}}\leq
1-{\delta'\over 2},
\end{align*}
provided $n$ is big enough. Thus,
$$
\esssup_{q} (\ident_{{\cal S}\setminus{\cal S}^*_{\delta}} h_n(q))\leq
\Big(1-{\delta'\over 2}\Big)^{m_n}\to 0,
$$
so that $\mu_n({\cal
S}^{*}_{\delta})\to 1$.
\\\\
Suppose now that $\mu(\bigcap_{\delta>0} \overline{\mathcal{S}_\delta^*} \setminus \mathcal{S}^*)=0$
and $\mu_n \Rightarrow \mu$. Given any $\rho \in \P$ and $A \subset \P$, define $d(\rho, A):=\inf_{\nu \in A} d(\rho, \nu)$ where $d$ is the Prokhorov metric. It is clear that $\rho \mapsto d(\rho, A)$ is a nonnegative, continuous function such that $d(\rho,A)=0$ if and only if $\rho \in \bar A$. Whence, $g_A(\rho):=\min(d(\rho, A), 1)$ is a nonnegative, bounded continuous function.
\\
Since $\mu_n \Rightarrow \mu$ we have
\[
 \int_\P g_{\mathcal{S}_\delta^*}(\rho) \mu(d\rho)=\lim_{n \to +\infty} \int_\P g_{\mathcal{S}_\delta^*}(\rho)\mu_n(d\rho) \le
 \lim_{n \to +\infty}
 \mu_n(\P \setminus\overline{\mathcal{S}_\delta^*}) =0,
\]
whence $\mu(g_{\mathcal{S}_\delta^*}=0)=1$,
that is $\mu(\overline{\mathcal{S}_\delta^*})=1$. By the continuity of the measure
$\mu$ and $\mu \big (\bigcap_{\delta>0} \overline{\mathcal{S}_\delta^*} \setminus \mathcal{S}^* \big )=0$, we have $\mu(\mathcal{S}^*)=1$.
\end{proof}
\\\\
Roughly speaking, according to Proposition \ref{point2}, in order to have $\mu({\cal S}^{*})=
1$, we need the weak convergence $\mu_n \Rightarrow \mu$. We shall now show that under some mild conditions, when
${\cal S}^*$ consists of one measure $q^*$ then the weak convergence of $\mu_n$ follows from $\mu({\cal S}^{*})=
1$ and the limit is $\delta_{q^*}$. In the following corollary, let
$B(q^*,\e)$ be an open ball in ${\cal P}$ centered at $q^*$
and having  radius $\e$.
\begin{corollary}\label{point} Let the assumptions of Proposition \ref{point2} hold.  Suppose ${\cal S}^*=\{q^*\}$. If, for any
$\e>0$ there exists $\delta>0$ such that
\begin{equation}\label{balls}
B(q^*,\e)\supseteq {\cal
S}^*_{\delta},
\end{equation}
then $\mu_n\Rightarrow \delta_{q^*}.$
\end{corollary}
\begin{proof} By hypothesis, for every
$\e>0$ there exists $\delta>0$ so that
$$\mu_n(B(q^*,\e))\geq \mu_n({\cal S}^{*}_{\delta}).$$ By Proposition \ref{point2}, $\mu_n({\cal S}^{*}_{\delta})\to 1$; whence, for every $\e>0$, $\mu_n(B(q^*,\e))\to 1$. This implies easily that
$\mu_n\Rightarrow \delta_{q^*}.$
\end{proof}
\paragraph{The result.}\label{par:result}
In this section, we consider a continuous $\phi$ having a unique minimum $x_o$. Thus, in what follows,
$$x_o:=\arg\inf_x \phi(x),\quad \phi_o:=\phi(x_o)=\inf_x \phi(x).$$
The following assumption is natural and assumes that  if
$\phi(x)$ is slightly bigger than the minimum $\phi_o$, then $x$ must be
close to $x_o$. It also guarantees that $x_o$ is the unique minimum and the convergence $\phi(x_n)\to \phi_o$ implies $x_n\to x_o$.
\begin{assumption}\label{ass:A}
For every $\e>0$  there exist $\delta>0$ so that
$\{x: \phi(x)-\phi_o\leq \delta\}\subset
B(x_o,\e)$, where $B(x_o,\e)$ is an open ball centered in $x_o$ and having radius $\e$.
\end{assumption}
Observe that if $\phi$ is continuous and $\X$ compact, then Assumption~\ref{ass:A} holds, since the minimum is unique.
\begin{theorem}\label{thm12} Let the fitness function be as in \eqref{wn} where $\phi$ is  non-negative, continuous, bounded and satisfies Assumption~\ref{ass:A};
and let $\lambda\in [0,1)$. Suppose $\pi_n=\pi$ and that the support of
$\pi$ contains the measure $\delta_{x_o}$. Then $\bp\Rightarrow
\delta_{q^*}$, where $q^*=\delta_{x_o}$. Moreover convergence \eqref{fi-con2} holds with the limit
 process being degenerate and having one almost sure path $(x_o,x_o,\ldots)$.
If, in addition, $\X$ is
compact, then $Q_n\Rightarrow \delta_{q^*}$.\end{theorem}
\begin{proof} Let us start with the case $\lambda \in (0,1)$ and define
$$f_n(q)=\langle \exp[-{\phi\over n^{\lambda}}], q
\rangle^{n^{\lambda}},\quad f(q)=\exp[-\langle \phi, q \rangle ].$$
Since $\phi$ is continuous and bounded, we obtain that $q\mapsto
\langle \phi, q \rangle$ is continuous and so is $f$. Also $f_n$ is continuous for every $n$.
 Assuming that
$\phi$ is bounded above, we obtain by Proposition \ref{Pee} that for
$\lambda>0$, $\|f_n-f\|_{\infty}\to 0$. Since $\lambda\in (0,1)$, we
take $m_n=n^{1-\lambda}$. Then
$$h_n(q):={\langle \exp[-{\phi\over n^{\lambda}}], q
\rangle^{n}\over Z_n}={f_n^{m_n}(q)\over Z_n}$$ so that $\mu_n=\bp$. Recall that ${\cal S}$ is the support of $\bp$. By definition
$${\cal S}^*= \arg\max_{q\in {\cal S}} f(q)=\arg\min_{q\in {\cal S}}\langle \phi, q \rangle.$$
Since $x_o$ is the unique minimum of $\phi$ and $q^*\in {\cal S}$,  then ${\cal
S}^*=\{q^*\}$ and  $${\cal
S}^*_{\delta}=\{q\in {\cal S}^*: \langle \phi, q \rangle-\phi_o\leq g(\delta)\}, \quad \textrm{where } g(\delta):=-\ln(1-\delta \exp(-\phi(x_o))) .$$
Proposition \ref{point2} applies and so for every $\delta>0$, it holds $\bp({\cal
S}^*_{\delta})\to 1$. In order to
apply Corollary \ref{point}, we have to check \eqref{balls}.  By the definition of Prokhorov metric,
the ball $B(q^*,\e)$ consist of all measures $q$ such that outside the
$\e$-ball $B(x_o,\e)$ the measure $q$ has mass at most
$\e$, so that
$$B(q^*,\e)=\{q\in {\cal P}: q\big(B(x_o,\e)\big)\geq
1-\e\}.$$ Therefore, we have to show the
following: for every $\e>0$, there exists $\delta>0$ such
that if a measure $q$ is such that $\langle \phi, q \rangle-\phi_o\leq
g(\delta)$, then it must hold that $q\big(B(x_o,\e)\big)\geq
1-\e$. Let $\e>0$ be fixed. By Assumption~\ref{ass:A}, there exists $\delta_o=\delta_o(\e)>0$ so that
$\{x: \phi(x)-\phi_o\leq \delta_o\}\subset B(x_o,\e)=:B.$ Take $\delta$ such that
${g(\delta) /\delta_o} \le \e.$
Suppose now that the measure $q$ satisfies $\langle \phi, q \rangle-\phi_o\leq
g(\delta)$. Then
$$
g(\delta)\geq \langle \phi, q \rangle-\phi_o = \int_\X(\phi-\phi_o)dq \geq  \int_{B^c}(\phi-\phi_o) dq\geq \delta_o (1-q(B)),
$$
whence $ q(B)\geq 1- {g(\delta) / \delta_o}=1-\e.$
By Corollary \ref{point}, $\bp\Rightarrow \delta_{q^*}$.\\\\
Consider the case $\lambda=0$. Take $f_n(q)=f(q)=\langle e^{-\phi}, q \rangle$. Since $\phi$ is continuous, then $f$ is continuous. The uniform convergence is trivial and Proposition \ref{point2} applies.
As previously, ${{\cal S}}^*=\{q^*\}$. We now have
$${\cal S}^*_{\delta}=\{q\in {\cal S}: \langle e^{-\phi}, q \rangle\leq e^{-\phi_o}+\delta\}=\{q\in {\cal S}: \langle w, q \rangle\leq w_o+\delta\},\quad w_o:=e^{-\phi_o}.$$
Observe that Assumption~\ref{ass:A} implies that the same holds for $w$: for every $\e>0$ small enough there exist $\delta>0$ so that
$\{x: w(x)-w_o\leq \delta\}\subset
B(x_o,\e)$. Hence, as before, \eqref{balls} holds, and so $\bp\Rightarrow \delta_{q^*}$.
Thanks to Corollary \ref{cor1},  the convergence $\bp\Rightarrow \delta_{q^*}$ implies that all the assumptions of Theorem \ref{thmP} are fulfilled and so the convergence \eqref{fi-con2} holds and the limit process is such that
with the limit only the fittest genotype survives.
\\
For the convergence $Q_n\Rightarrow \delta_{q^*}$, we apply Theorem \ref{thmQ} and Corollary \ref{cor1}: when $\lambda\in (0,1)$, then $w\equiv 1$, and so $Q_n\Rightarrow \delta_{q^*}$.
 When $\lambda=0$, then $w_n=w=e^{-\phi}$. Since $\phi$ is bounded, then $w$ is bounded away from 0. Recall that $r: \P\to \P$ is as follows $r_q(A)=\langle w,q \rangle^{-1}\int_A w dq$.
Thus $r_{q^*}=\delta_{x_o}=q^*$. Since $\p=\delta_{q^*}$, it holds that $Q=\p r^{-1}=\delta_{r_{q^*}}=\delta_{q^*}$.
\end{proof}
\section{Dirichlet process priors}\label{sec:dir}
In this section we consider the Dirichlet process priors as follows:
$\pi_n$ is the distribution of Dirichlet process $D(\a_n)$, where
$\a_n$ is a finite measure on $\B(\X)$ (the base measure). Recall
that a random measure $P$ on $(\X,\cal B(X))$ possesses a Dirichlet
process distribution $D(\a_n)$, when for every finite measurable
partition $A_1,\ldots,A_k$ of $\X$,
$$\big(P(A_1),\ldots,P(A_k)\big)\sim {\rm Dir}\big(k; \a_n(A_1),\ldots,\a_n(A_k)\big),$$
where ${\rm Dir}\big(k; \a_n(A_1),\ldots,\a_n(A_k)\big)$ stands for the
$k$-dimensional Dirichlet distribution with parameters
$\big(\a_n(A_1),\ldots,\a_n(A_k)\big)$. As it is common, we write
$m_n:=\a_n(\X)$ for the total mass of the base measure, and define a
probability measure $\ba_n:=\a_n/m_n$.  In what follows, we assume that
$\bar \a_n$ is fixed and independent of $n$, thus $\ba_n=\ba$, but $m_n$
depends on $n$, and is typically increasing in $n$. We also use the
parallel notation $DP(m_n,\ba)$. Since for any $A\in \B(\X)$,
$E[P(A)]=\ba(A)$ and ${\rm Var}[P(A)]={\ba(A)(1-\ba(A))\over
m_n+1}$, we see that the bigger $m_n$, the more is the process
concentrated on its mean $\ba$. Therefore increasing $m_n$ means
increasing the influence of the prior. Considering the Dirichlet process
as a random element on  a probability space $(\Omega,{\cal F},{\bf
P})$, i.e. $P: \Omega\to \P$, we define the measure $\pi_n$ as its
distribution $\pi_n(E):={\bf P}(P\in E)$, $E\in \B(\P)$. We shall refer to that prior as the { Dirichlet process prior} $DP(\a_n)$ or $DP(m_n,\ba)$.
\subsection{Fixed mutation probability and fixed fitness ($\lambda=0$)}
\label{main-lambda0}
Let us now consider the case where $m_n=c\cdot n$
where  $c>0$. Recall that $c$ determines the mutation probability ${c\over c+1}$.   We assume that
the support of $\ba$ is $\X$, and
then also the support of $\pi_n$ is $\P$. Recall that in this case the fitness function  \eqref{wn} is
$w(x)=e^{-\phi(x)}$,
where $\phi$ is nonnegative and continuous,
({hence bounded when $\X$ is compact}). We also assume the existence of $x_o$ such that $\phi(x_o)=\inf_{x\in \X}\phi(x)=:\phi_o$. It means that $w(x_o)=\sup_{x\in \X}w(x)$.
For the time being, $x_o$ need not to be unique; sometimes the uniqueness is needed, and then we specify it later.
\\\\
By Theorem \ref{thmP} and Corollary \ref{cor1}, to prove convergence of $P_n$ we just need to prove convergence of
$\bar\pi_n$ (which was defined in \eqref{eq:barpin}).
The main idea is to prove that the sequence $\bar\pi_n$ satisfies a LDP with a certain rate function $I(q)$ (whence the need of compactness which is a usual assumptions in large deviations theory). The rate function can be written as
$I(q)=\sup_{q^\prime\in\P} F(q^\prime)-F(q)$. The expression of $F$ is different in the case $\lambda=0$ and $\lambda>0$, but in
both cases we prove that if $w$ has a unique maximizer, so has $F$. Then $I(q)$ is positive, but for one measure $q^*$ for which $I(q^*)=0$ and from the LDP we get convergence of $\bar\pi_n$ to $\delta_{q^*}$. The details are rather lengthy and have been put in
the Appendix.
\\\\
In what follows, for any measurable positive function $f$, we shall write $\langle w, f \rangle=\int_\X wf d\ba$ instead of using the cumbersome notation $\langle w, \mu_f\rangle$ where $\frac{d\mu_f}{d\bar \a}=f$.
\begin{lemma}\label{lemma1b}  If the following inequality holds:
\begin{equation}\label{marta}
\int_\X \frac{w(x_o)}{w(x_o)-w(x)}\ba (dx)\geq {1+c\over
c},
\end{equation}
then there exists only one $\theta\geq {{w}(x_o)\over c+1}$ such
that
\begin{equation}\label{tihedus2} f(x):={c\over \big(1+c-{w(x)\over \theta}\big)}
\end{equation}
 is a probability density with
respect to $\ba$. This unique $\theta$ satisfies the (implicit) equation $\theta=\langle w, f \rangle$.
\end{lemma}
\begin{proof} Clearly $f(x)\geq 0$ if and only if $\theta\geq
\frac{w(x_o)}{c+1}
=:\theta_o$. Denoting $f$ by $f_{\theta}$, we see
that $\theta \mapsto \int_\X f_{\theta}d\ba$ is continuous and strictly
decreasing. Hence, there exists at most one $\theta$ such that
$f_{\theta}$ is a probability density. Such  $\theta$ exists, if $\int_\X
f_{\theta_o}d\ba\geq 1$, which is equivalent to \eqref{marta}. To
see that $\theta=\langle w, f \rangle$, observe
that, when $f_\theta$ is a probability density,
$$c=\int_\X \big(1+c-{w\over \theta}\big)f_\theta d\ba=1+c-{\int wf_\theta d\ba\over \theta},$$
which is possible only if $\theta=\langle w, f_\theta \rangle$.
\end{proof}
\\\\
It turns out that the inequality \eqref{marta} is crucial. It does
not hold, when $w$ has a very sharp peak around its maximum value or
$\ba$ puts very little mass around the maximum of $w$. Hence
\eqref{marta} somehow characterizes $w$ as well as $\ba$. Observe that when \eqref{marta} fails, then $\ba(w^{-1}(w(x_o))=0$;
in particular, it cannot have an atom at $x$ whenever $w(x)=w(x_o)$.
\\\\
In order to state and prove our main results we define two subsets of $\P$, namely
$$\P_1:=\{q\in \P: \ba \ll q\}, \quad {\cal P}_o :=\{q\in \P_1: q\ll \ba\}.$$
Note that, $q\in \P_o$ if and only if there exists a measurable function $h$ such that
$q(h>0)=1$ and
$d\bar \a/dq=h$.
In this case $q(E)=0$ if and only if $\bar \a(E)=0$; moreover
$dq/d\bar\a=1/h$.
\\\\
We will see that the asymptotic distribution has a different shape according to whether or not equation~\eqref{marta} holds. On the one hand, when \eqref{marta} holds we have a probability density given by equation~\eqref{tihedus2} and we denote by $q^*$ the corresponding measure; clearly
{$q^* \in \P_0$}. On the other hand,
when \eqref{marta} fails, suppose that
$x_o$ is one of the absolute maxima of $w$
(right now, we do not assume $x_o$ to be unique), then $\bar \alpha(x_o)=0$;
since \eqref{marta} fails, it holds
\begin{equation}\label{tihedus3}
\beta:=\int_\X f d\ba<1,\quad \text{where}\quad f:=\frac{cw(x_o)}{(1+c)(w(x_o)-w(x))}
\end{equation}
and, in this case, we define the measure
\begin{equation}\label{q-def}
q^*:=\beta q^a+(1-\beta)\delta_{x_o},
\end{equation}
where $q^a$ has density  $\beta^{-1}f$ (with respect to $\ba$); we
{shall argue in Appendix (after (\ref{eq:thetao})) that in this case $q^*\in \P_1 \setminus \P_o$}.
\\\\
Our goal is to prove that if $x_o$ is the unique maximum for $w$ then $\bp\Rightarrow \delta_{q^*}$.
Recall that in our case
\begin{equation}\label{measure}
\bp(E)={1\over Z_n}\int_E \langle  w,q \rangle^n \pi_n (dq)={1\over
Z_n}\int_E \exp [n \ln \big(\langle  w,q \rangle\big)]  \pi_n(dq)= {1\over Z_n}\int_E\exp[m_n G(q)]\pi_n(dq),\end{equation}
where  $G(q):={1\over c}\ln(\langle w,q
\rangle)$ and $m_n:=c\cdot n$. When $w$ is bounded and continuous, then $G$
is a continuous function on $\P$. We use the following theorem:
\begin{theorem}\label{thmFeng} \text{\rm (Corollary 9.3 \cite{Feng})} Let $\X$ be compact, $G:
\P\to \mathbb{R}$  be a continuous function, and let
$\pi_n$ be the $DP(m_n,\ba)$-prior. Define the sequence of measures
$$\bp(dq)={1\over Z_n}\exp[m_n G(q)]\pi_n(dq).$$
The sequence satisfies a Large Deviation Principle (in short LDP) on the space $\P$ as $n$ tends to infinity,
with speed $m_n^{-1}$ and rate function
$$I(q)=\sup_{q'\in \P}[G(q')-D(\ba\|q')]-(G(q)-D(\ba\|q)).$$
\end{theorem}
Note that the following theorem holds not only with the weights $w$ as in \eqref{wn} but also for general continuous (hence bounded),
non negative weights.
\begin{theorem}\label{thm1a}
Let $\X$ be compact
and let  $x_o$ be the unique maximum for $w$
If \eqref{marta} holds, define  $q^*$ as the measure $fd\ba$, where $f$ is as in \eqref{tihedus2};
otherwise define $q^*$ as in \eqref{q-def}. Let $r_{q^*}$ be the  probability measure defined on $ \B(\X)$,
such that $r_{q^*}(A)\propto \int_A w dq^*$, for all $A\in \B(\X)$.
 Then
 \begin{enumerate}
  \item $\bp\Rightarrow \delta_{q^*}$;
  \item the limit process of $P_n$ (in the sense of \eqref{fi-con2}) is an i.i.d. process where $X_i\sim r_{q^*}$;
  \item $Q_n\Rightarrow \delta_{r_{q^*}}$.
 \end{enumerate}
%
\end{theorem}
See Appendix for the proof.


\paragraph{Example.} Take $\X=[0,1]$, $x_o=0.3$, $\ba$ -- Lebesgue measure and $\phi(x)=|x-x_o|^p$. Then \eqref{marta}) holds, if
\begin{equation}\label{marta1a}
\int_0^1 \Big({e^{|x-0.3|^p}\over
e^{|x-0.3|^p}-1}\Big) dx\geq {c+1\over c}.
\end{equation}
When  $c$ is sufficiently big and $p$ is small enough and, i.e. $p<p^*(c)$
(for instance $p(1)\approx 0.2$), then \eqref{marta1a} fails.\\\\
When $p>p^*$ (so that \eqref{marta1a} holds), then there exists $\theta \in [{1\over 1+c},1]$
$w(x_o)=1$) so that
$$f(x)=c\Big(1+c-{e^{-|x-0.3|^p}\over
\theta}\Big)^{-1}$$
would integrate to 1 over $[0,1]$. Then $r_{q^*}$ has density
$${w(x)f(x)\over \theta}={c \exp[-|x-0.3|^p]\over \theta (1+c)-\exp[-|x-0.3|^p]}={c\over (1+c)\theta e^{|x-0.3|^p}-1}=:f^*(x).$$
Thus, when $n$ is big enough and $X_1,\ldots,X_n\sim P_n$, then $X_1,\ldots,X_n$ are approximatively i.i.d. with density $f^*$.\\\\
When $p<p^*$ (\eqref{marta1a} fails), then $\theta=\theta_o=1/(c+1)$ and so the function $f$ in \eqref{tihedus2} is
$$f(x)={c\over 1+c}\Big(1-\exp[-|x-0.3|^p]\Big)^{-1},\quad \beta=\int_0^1 f(x)dx<1. $$
Thus $${f(x)w(x)\over \theta}={c\over  e^{|x-0.3|^p}-1}=:f^a(x).$$
Hence the measure $r_{q^*}$ is such that
$$r_{q^*}(A)=\int_A f^a(x)dx+(1-\beta)(1+c)\delta_{0.3}(A),$$
i.e. it has absolutely continuous part with density $f^a$ (integrating up to $\beta(1+c)-c$) and atom $0.3$ with mass $(1-\beta)(1+c)$.
Thus, when $n$ is big enough and $X_1,\ldots,X_n\sim P_n$, then $X_1,\ldots,X_n$ are approximatively i.i.d. with measure $r_{q^*}$.


\subsection{The case $\lambda\in (0,1)$}
\label{main-lambda01}
The authors in \cite{article1} consider also the case when $\pi_n$ is the $DP(n^{1-\lambda},\ba)$-prior and the fitness function is as in \eqref{wn}, i.e. $w_n=\exp[-{\phi(x)\over  n^{\lambda}}]$,
$\lambda \in (0,1)$. Thus the mutation probability is $(1+n^{\lambda})^{-1}\to 0$. When $\X$ is finite, then in this case the limit process is again an i.i.d process with some measure $q^*$ that differs from the measures $r_{q^*}$ in the case $\lambda=0$. Somehow
surprisingly, the limit measure is independent of $\lambda$.
We still assume the existence of $x_o$ such that $\phi(x_o)=\inf_{x\in \X}\phi(x)=:\phi_o$, so that $w(x_o)=\sup_{x\in \X}w(x)$.
\\\\
We now take $\pi_n$ as $DP(c\cdot n^{1-\lambda},\ba)$-prior. Then the mutation probability is $c/(n^{\lambda}+c)$ and we
 need to study the measure $\bp$, where
\begin{align}\label{moot}
\bp(E)&={1\over Z_n}\int_E \langle w_n,q \rangle^n \pi_n(dq)={1\over Z_n}\int_E \exp[n \ln \langle w_n,q \rangle] \pi_n(dq)\\
&=
{1\over Z_n}\int_E \exp[n^{1-\lambda} ( n^{\lambda} \ln \langle w_n,q \rangle )] \pi_n(dq)={1\over Z_n}\int_E \exp[m_n \cdot  G_n(q)] \pi_n(dq).\end{align}
 Here $$m_n:=c\cdot n^{1-\lambda},\quad G_n(q):={1\over c}\cdot n^{\lambda} \ln \langle w_n,q \rangle={1\over c}\cdot\ln \big( \langle w_n,q \rangle^{n^{\lambda}}\big).$$
We see that $G_n$ depends on $n$, and so Theorem \ref{thmFeng} does not immediately apply. On the other hand, when $\phi$ is bounded,
then  by Proposition \ref{Pee}
$$\langle w_n,q \rangle^{n^{\lambda}}\to \exp[-\langle \phi,q\rangle ]$$
and the convergence is uniform in $q\in\P$. Since $q\mapsto \exp[-\langle \phi,q\rangle ]$ is bounded, it follows that $G_n(q)\to G(q):=-{1\over c}\langle \phi,q\rangle $ uniformly.
\\\\
The following result is the analogous of Lemma~\ref{lemma1b}; the proof is similar and is left to the reader.
\begin{lemma}\label{lemma1bb} Consider the bounded function $\phi:\X\to [0,\infty)$ and define $\phi_o:=\inf_x \phi(x)$. 
If the following inequality holds:
\begin{equation}\label{marta2}
\int_\X {1\over \phi(x)-\phi_o}\ba (dx)\geq {1 \over c},
\end{equation}
then there exists one $\theta \in [\phi_o, \phi_o+c]$ so that
\begin{equation}\label{tihedus4}
f(x)={c\over \phi(x)+c-\theta}
\end{equation}
is a probability density with respect to $\ba$, and then $\theta=\langle \phi,f \rangle=\int_\X \phi(x)f(x)\ba(dx).$
\end{lemma}
Again, when \eqref{marta2} fails, then $\ba(\{x\colon \phi(x)=\phi_o\})=0$. As in the previous case, the limit distribution takes two completely different shapes according to whether or not equation~\eqref{marta2} holds. If it holds we have a probability density given by \eqref{tihedus4} and we denote by $q^*$ the corresponding measure; clearly
{$q^* \in \P_0$}. If not,
consider $x_o$ such that $\phi(x_o)=\phi_o$,
\begin{equation}\label{fo}f(x)={c\over \phi(x)-\phi_o},\quad \beta:=\int_\X f d\ba<1,\end{equation}
and define \begin{equation}\label{q-def2}
q^*=\beta q^a+(1-\beta)\delta_{x_o},\end{equation}
where $q^a$ has density $\beta^{-1}f$ with respect to $\bar \a$ and $\phi(x_o)=\phi_o$.
{Since $f>0$ everywhere (due to the boundedness of $\phi$), it follows $q^* \in \P_1 \setminus \P_o$}.
Observe that $q^*$ is independent of $\lambda$. Hence, under \eqref{marta2},
it has density  \eqref{tihedus4} with respect to $\ba$. Otherwise the measure $q^*$ has atom $x_o$ that has mass $(1-\beta)$, where $\beta$ is defined as in \eqref{fo}.
\\\\
The following theorem generalizes Theorem \ref{thmFeng} (see Section~\ref{main-lambda01appendix} for the details of the proof).  
\begin{theorem}\label{thmFeng2} Let $\X$ be compact, $G,G_n:
\P\to \mathbb{R}$  be continuous functions that converge uniformly: $\sup_q |G_n(q)-G(q)|\to 0$, and let
$\pi_n$ be $DP(m_n,\ba)$-prior. Define the sequence of measures
$$\bp(dq)={1\over Z_n}\exp[m_n G_n(q)]\pi_n(dq).$$
The sequence satisfies a LDP on the space $\P$ as $n$ tends to infty,
with speed $m_n^{-1}$ and rate function
$$I(q)=\sup_{q'\in \P}[G(q')-D(\ba\|q')
]-(G(q)-D(\ba\|q)
).$$
\end{theorem}
The following is the analog of Theorem~\ref{thm1a} in the case $\lambda\in (0,1)$.
\begin{theorem}\label{thm1b} Let $\X$ be compact and  $\phi$ continuous. Assume that  $x_o$ is
the unique maximum of $w$.
If \eqref{marta2} holds, define  $q^*$ as the measure $fd\ba$, where $f$ is as in \eqref{tihedus4};
otherwise define $q^*$ as in \eqref{q-def2}.
 Then
 \begin{enumerate}
  \item $\bp\Rightarrow \delta_{q^*}$;
  \item the limit process of $P_n$ (in the sense of \eqref{fi-con2}) is an i.i.d. process where $X_i\sim {q^*}$;
  \item $Q_n\Rightarrow \delta_{{q^*}}$.
 \end{enumerate}
\end{theorem}
The proof can be found in Appendix.


\paragraph{Example.} Take $\X=[0,1]$, $x_o=0.3$, $\ba$ -- Lebesgue measure and $\phi(x)=|x-x_o|^p$. Then \eqref{marta2} holds, if
\begin{equation}\label{marta2a}
\int_0^1 |x-x_o|^{-p}dx \geq {1\over c}.
\end{equation}
When it is so then limit process is iid process, $X_1,X_2,\ldots$ where $X_i$ has density
$$f(x)={c\over |x-x_o|^p+c-\theta},\quad 0<\theta\leq c.$$
Otherwise   $X_1,X_2,\ldots$ is iid process with
$$P(X_i\in A)=c\int_A |x-x_o|^{-p}dx+(1-\beta)\delta_{x_o}(A),\quad \text{where     }\beta=c\int_0^1 |x-x_o|^{-p}dx.$$

\section{Appendix}

\subsection{The case $\lambda=0$.}

In this section we collect all the technical results that we need to prove the main theorems of Section~\ref{main-lambda0}.
Recall that $\P$ stands for the set of all probability measures on ${\cal B}(\X)$ and remember the definitions
$$\P_1:=\{q\in \P: \ba \ll q\}, \quad {\cal P}_o :=\{q\in \P_1: q\ll \ba\}.$$
Thus $\P_o\subset \P_1$. We define the objective function $F$ on $\P_1$ as follows:
\begin{equation}
\label{DEF}F(q):=\ln \langle w,q\rangle -c D(\ba\|q),
\end{equation}
where the relative entropy $D(\ba\|q)$ is defined by equation~\eqref{eq:entropy}.
We now observe that on the set $\P_o$,
$$F(q)=\ln \langle w,g\rangle +c\int_\X \ln g d\ba,\quad \text{where}\quad g={dq\over d\ba}\quad \text{and} \quad\langle w,g\rangle=\int_\X w g d\ba. $$
Indeed, $q\in \P_o$ if and only if there exists a measurable function $h$ such that
$q(h>0)=1$ and
$d\bar \a/dq=h$.
In this case $q(E)=0$ if and only if $\bar \a(E)=0$; moreover
$dq/d\bar\a=1/h$.
It follows that when $q\in \P_0$ then
$$
D(\ba\|q)=\int_\X \ln \frac{1}{g} d\ba=-\int_\X \ln g d\ba.
$$
Therefore, if $q\in \P_0$, then
$$F(q)=\ln \big(\int_\X w dq \big)-c \int_\X \ln h d\ba=\ln \big(\int_\X w g d\ba  \big)+c \int_\X \ln g d\ba.
$$
When $q \in \P_o$, in order to stress the dependence of $F(q)$ on $g=dq/d\bar\alpha$, with a slight abuse of notation, we will write
$F(g)$ instead of $F(q)$.
\\\\
In what follows, we are interested in maximizing $F(q)$ over $\P_1$
and finding the argmax, when it exists. We split the maximization problem into two parts: maximizing over $\P_o$ and $\P_1\setminus \P_o$. On $\P_o$, it holds $F(q)=F(g)$, where $g$ is the density of $q$ with respect to $\ba$. Hence
$$\sup_{q\in \P_o} F(q)\leq \sup_{g\in \F} F(g),$$
where ${\cal F}$ is the set of all probability densities with respect to $\ba$. Here there is an inequality, because the definition of $\P_o$ implies that $\supp g=\X$, but $\F$ is the set of all probability densities.
We start with maximizing $F$ over $\F$.
\subsubsection{Maximizing $F$ over $\F$}
\paragraph{Inequality \eqref{marta} holds.}
We now show that under \eqref{marta} the function
$f$ as in \eqref{tihedus2} is the unique solution of the
above-stated maximization problem. Observe that $F(f)<\infty$. Indeed, since ${c\over (c+1)}\leq f(x)\leq {c\over (c+1)-
\frac{w(x_o)}{\theta}}
$ then $\ln f(x)$ is bounded from below; moreover $\ln f(x)\leq f(x)-1$.
\begin{lemma}\label{lemma1} If \eqref{marta} holds, then $\sup_{g\in \F} F(g)=F(f)$ and $F(g)<F(f)$ when $g\ne f$, where $f$ is given by  \eqref{tihedus2}.
\end{lemma}
\begin{proof} Let $f$ be the density \eqref{tihedus2}. It suffices to show that for any other $g\in {\cal F}$ such that  $g\ne f$  $\ba$-a.s.
\begin{equation}
F(f)-F(g)=\ln \theta + c\int_\X \ln f d\ba  -\ln \theta' - c\int_\X \ln g d\ba = \ln {\theta\over \theta'} - c \int_\X \ln {g\over f} d\ba >0,
\end{equation}
where  $\theta'=\langle w,g \rangle$. When $\int_\X \ln g d\ba=-\infty$, then the strict inequality holds, otherwise observe that all integrals above are finite and since
$$\int_\X {g\over f} d\ba={1+c\over c}\Big(1-{\theta'\over \theta (c+1)}\Big),$$
by Jensen inequality we get
$$- c \int_\X \ln {g\over f} d\ba > - \ln \Big[{1+c\over c}\Big(1-{\theta'\over \theta (c+1)}\Big)\Big]^c,$$
where the strict inequality follows from assumption $f\neq g$ $\ba$ - a.s.. Therefore, it suffices to show that
$$\ln {\theta\over \theta'}- \ln \Big[{1+c\over c}\Big(1-{\theta'\over \theta (c+1)}\Big)\Big]^c\geq 0.$$
The latter is equivalent to
\begin{equation}\label{jy}\ln \Big[{1+c\over c}\Big(1-{\theta'\over \theta (c+1)}\Big)\Big]^c-\ln {\theta\over \theta'}\leq 0\quad \Leftrightarrow \quad \Big[{1+c\over c}\Big(1-{\theta'\over \theta (c+1)}\Big)\Big]^c
\cdot  {\theta'\over \theta}\leq 1.\end{equation}
Since $
\frac{w(x_o)}{c+1}
\leq \theta=\langle w,f \rangle\leq
w(x_o)$,
 and $\theta'=\langle w,g \rangle\leq
w(x_o)$,
it holds that
$$0\leq {\theta'\over \theta (c+1)}\leq
\frac{w(x_o)}{w(x_o)}
=1.$$
Denoting ${\theta'\over \theta (c+1)}=:1-\alpha$, we obtain that the right hand side inequality in \eqref{jy} is
$$\Big[\big(1+{1\over c}\big)\alpha\Big]^c(c+1)(1-\alpha)\leq 1$$
and this holds by Proposition \ref{append}, proven below.\end{proof}

\begin{proposition}\label{append}
\begin{equation}\label{xa}\max_{x\geq 0,\a\in [0,1]}\Big((1+x)(1+{1/x})^x \a^x(1-\a)\Big)=1.
\end{equation}
\end{proposition}
\begin{proof} Fix $\a\in (0,1)$. Then
$x_{\alpha}={\a\over 1-\a}$ attains the maximum of
$$u(x):=(1+x)(1+{1/x})^x \a^x$$
over $[0,\infty)$. To see that observe: $\lim_{x\to \infty}u(x)=0$, $\lim_{x\to 0}u(x)=1$ and
$$u'(x)={(\ln u(x))' u(x)}=\big(\ln(1+{1/x})+\ln \a\big)u(x),$$
so that $x_{\a}$ is the only stationary point of the function $u(x)$. Plugging $x_{\a}$ into the left hand side of \eqref{xa}, we obtain
$$(1+x_{\a})(1+{1/ x_{\a}})^{x_{\a}} {\a}^{x_{\a}}(1-\a)=1.$$
\end{proof}
\paragraph{Inequality \eqref{marta} fails.} Let $x_o$ be such that $w(x_o)=\sup_xw(x)=:\bar{w}$.
Right now, we do not assume $x_o$ to be unique.
Since \eqref{marta} fails, we define $\beta$, $f$ and $q^*$ as in \eqref{tihedus3} and \eqref{q-def}.
Clearly
\begin{equation}\label{eq:thetao}
\langle w,q^*\rangle=\langle w,f\rangle+(1-\beta)
\bar{w}=\bar{w}
\big(\beta-{c\over 1+c}+1-\beta\big)=
{\bar{w}\over 1+c}
=\theta_o;
\end{equation}
recall that $\theta_o$ was defined in the proof of Lemma~\ref{lemma1b}.
 Observe that
 $q^* \in \P_1 \setminus \P_o$. Indeed, since
 \eqref{marta} fails, then $0=\bar\a(w^{-1}
(\bar w))
\ge \bar\a(x_o)$ while $q^*(x_o)>0$.
 On the other hand,
{since $f(x)>0$ whenever $x\not\in w^{-1}(\bar{w})$, it follows that
  $q^*(E)=0$ implies $q^a(E)=0$ whence $\bar \a(E)=0$.} In this case
 ${d\ba}/{dq^*}=\ident_{\X\setminus\{x_o\}}/f$.
Now, according to \eqref{DEF}
\begin{equation}\label{Fdef}
F(q^*)
=\ln \theta_o+c \int_\X \ln f d\ba.\end{equation}
\begin{lemma}\label{lemma2} Let \eqref{marta} fail and $g\in {\cal F}$. Then $F(g)<F(q^*)$, where $F(q^*)$ is defined as in \eqref{Fdef}.
\end{lemma}
\begin{proof} Observe that $F(q^*)$ is independent of the choice of $x_o$. The proof is the exactly same as that of Lemma \ref{lemma1}. The Jensen inequality is an equality, when $f/g$ is constant $\ba$-a.s.. In our case it means
 $g=\beta^{-1}f$. In this case
 \begin{align*}
 F(g)& =\ln ( \langle w,f\rangle )-\ln \beta +c\int_\X \ln f d\ba-c\ln \beta=\ln \big(
\bar{w}
(\beta-{c\over 1+c})\big)-(c+1)\ln \beta +c \int_\X \ln f d\ba \\
 &< \ln \big(
{\bar{w}\over 1+c}\big)+c \int_\X \ln f d\ba=F(q^*).\end{align*}
 The inequality holds because,
$$ \ln \big(\bar{w}(\beta-{c\over 1+c})\big)-(c+1)\ln \beta < \ln \big({\bar{w}\over 1+c}\big)\quad \Leftrightarrow \quad \beta(1+c)-c<\beta^{c+1}$$
and for $\beta <1$ the L.H.S.~holds.
 \end{proof}
\subsubsection{Maximizing $F$ over $\P_1\setminus\P_o$}
\paragraph{Inequality \eqref{marta} holds.} Let $f$ be the density \eqref{tihedus2}.
\begin{lemma}\label{lemma3} Let $q\in \P_1\setminus\P_o$. Then $F(q)<F(f)$.\end{lemma}
\begin{proof} Let $h={d\ba\over dq}$. Let $q(h>0)=:\beta_1$. Since $q\not\in \P_o$, $\beta_1<1$, because otherwise $1/h$ would be a density. Clearly $\beta_1>0$. Let $g:=1/h$. Thus $q(g=\infty)>0$, but $\ba(g=\infty)=\ba(h=0)=0$.
For any Borel set $B$
$$\int_B  gd \ba=\int_{B \cap \{h>0\}} gd \ba= \int_{B \cap \{h>0\}} {1\over h} h dq=\int_{B\cap \{h>0\}}dq=q(B\cap \{h>0\}).$$
Whence
\begin{equation}\label{eq:fundamentalq}
 q(B)=\int_B g d\bar \a+q(\{h=0\} \cap B)
\end{equation}
which, in particular, implies $q(h=0)=1-\beta_1$.\\
We also get that $$\int_\X g d\ba=q(h>0)=\beta_1,\quad \int_{\{h>0\}}w dq=\int_{\{h>0\}}w g d\ba=\int_\X wg d \ba,$$
because $\ba(h=0)=0$ and so
\begin{equation}\label{kpool}
\theta'=\langle w,q\rangle=\int_{\{h>0\}}wg d\ba+\int_{\{h=0\}}wdq\leq \int_\X wg d\ba+\bar{w}(1-\beta_1)=\langle w,g\rangle+\bar{w}(1-\beta_1),\end{equation}
where the equality holds if and only if $
q(h=0)=q(h=0, \, w=\bar w)\equiv 1-\beta_1$.
By definition \eqref{DEF},
\begin{equation}\label{muhvik}
F(q)=\ln[\langle w,q\rangle]-c\int_\X  \ln h d\ba = \ln \theta'+c\int_\X \ln g d\ba.\end{equation}
Therefore, it suffices to prove that
\begin{equation} F(f)-F(q)=\ln \theta+c\int_\X \ln f d\ba - \ln \theta'- c\int_\X \ln g d\ba= \ln {\theta\over \theta'} - c \int_\X \ln {g\over f} d\ba >0.\end{equation}
The proof follows the steps   of Lemma \ref{lemma1}. In this case
\begin{equation}\label{shabe}\int_\X {g\over f} d\ba={1+c\over c}\Big(\beta_1-{\langle w,g\rangle \over \theta (1+c)}\Big)\end{equation} so that, after using the Jensen inequality, instead of the inequality \eqref{jy}, now we have
\begin{equation}\label{jy2}
\Big[{1+c\over c}\Big(\beta_1-{\langle w,g\rangle \over \theta (1+c)}\Big)\Big]^c{\theta'\over \theta}\leq 1.
\end{equation}
To see that \eqref{jy2} holds, define
$$\a:=\beta_1-{\langle w,g\rangle \over \theta (1+c)}.$$
Since $\langle w,g\rangle \leq \bar{w}\beta_1$ and $\theta(1+c)\geq \bar{w}$, it holds that $\a\in[0,\beta_1]$. On the other hand, by \eqref{kpool}
$${\theta'\over \theta}\leq {\langle w,g\rangle\over \theta}+{\bar{w}(1-\beta_1)\over \theta}\leq (\beta_1-\a)(1+c)+(1+c)(1-\beta_1)=(1+c)(1-\a).$$
Hence
$$\Big[{1+c\over c}\Big(\beta_1-{\langle w,g\rangle \over \theta (1+c)}\Big)\Big]^c{\theta'\over \theta}\leq (1+c)\big(1+{1\over c}\big)^c\alpha^c(1-\a)\leq 1,$$
where the last inequality comes from Proposition \ref{append}. This proves
the strict inequality if the Jensen inequality is strict.
\\
Since $\int_\X g d\ba=\beta_1<1$, we obtain that the Jensen inequality is an equality if and only if $f=\beta_1^{-1}g$; in this case by \eqref{kpool}, since $\langle w,g\rangle=\beta_1\langle w,f\rangle =\beta_1 \theta$, whe have
$\theta'\leq \theta (\beta_1+(1-\beta_1){\bar{w}\over \theta})$ and $\int_\X \ln {g\over f} d\ba=\ln (\beta_1)$ so that
$$\ln{\theta\over \theta'} - c \int_\X \ln {g\over f} d\ba\geq -\ln(\beta_1+(1-\beta_1){\bar{w}\over \theta})-c \ln \beta_1\geq -\ln(\beta_1+(1-\beta_1)(1+c))-c \ln \beta_1 >0,$$
because for every $c>0$, it holds
$$\beta_1^c(\beta_1+(1-\beta_1)(1+c))=\beta_1^c(1-c\beta_1+c)<1.$$
\end{proof}
\paragraph{Inequality \eqref{marta} fails.}  \begin{lemma}\label{lemma4} Assume that $x_o$ is unique. Let $q\in \P_1\setminus\P_o$ and $q\ne q^*$, where $q^*$ is defined as in \eqref{q-def}. Then $F(q)<F(q^*)$.\end{lemma}
\begin{proof} We know that, since \eqref{marta} fails, then $\ba(x_o)=0$, whence $q^*\in \P_1\setminus \P_o$ and since $x_o$ is unique, the construction \eqref{q-def} uniquely defines $q^*$ as well.
We have to prove
\begin{equation}\label{neljas} F(q^*)-F(q)=\ln {\theta_o\over \theta'} - c \int_\X \ln {g\over f} d\ba >0,
\end{equation}
where
$$\theta_o={\bar{w}\over 1+c},\quad h={d\ba\over dq},\quad  g=h^{-1},\quad \beta_1=\int_\X g d\ba,\quad \theta'=\langle w,q\rangle\leq \langle w,g\rangle +(1-\beta_1){\bar w}$$
and $f$ is as in \eqref{tihedus3}. Observe $\int_\X f d\ba =\beta={\langle w,f \rangle\over \bar{w}}+{c\over 1+c} \geq {c\over 1+c}$. Again, we follow the steps of Lemma \ref{lemma1}. Since now $\theta_o (1+c)={\bar w}$, from \eqref{shabe} we get
\begin{equation}\label{shabe2}\int_\X {g\over f} d\ba={1+c\over c}\Big(\beta_1-{\langle w,g\rangle \over {\bar w}}\Big)\end{equation}
and with $\a=\beta_1-{\langle w,g\rangle \over {\bar w}}$ and as in the proof of Lemma \ref{lemma3} we obtain that \eqref{jy2} holds. \\\\
When the Jensen inequality is strict, there is nothing to prove.
Observe that the Jensen inequality is an equality if and only if $g={\beta_1\over \beta}f$ $\ba$-a.s. So, when $\beta=\beta_1$, it means that $f=g$. This implies that the inequality
$\theta'\leq \langle w,g\rangle +(1-\beta_1){\bar w}=\langle w,f\rangle +(1-\beta){\bar w}=\theta_o$ must be strict. Indeed, otherwise by \eqref{kpool}, $q(h=0)=q(h=0,\, w=\bar w)=1-\beta_1$ but $w^{-1}(\bar w)=\{x_o\}$, thus $q(x_o)=1-\beta_1=1-\beta=q^*(x_o)$ whence $q=q^*$.
But when  $\theta'<\theta_o$, and $g=f$, then \eqref{neljas} trivially holds.
\\
Let us now consider the case $\beta_1 \ne \beta$ and $g={\beta_1\over \beta}f$ $\ba$-almost surely.
In this case
$$\langle w,g\rangle={\beta_1\over \beta}\langle w,f\rangle={\beta_1\over \beta}\big(\beta-{c\over 1+c}\big)\bar{w}$$
so that $$\theta'\leq \langle w,g\rangle + (1-\beta_1)\bar{w}=\bar{w}\big(1-{\beta_1\over \beta}{c\over 1+c}\big)=\bar{w}\big({\beta(1+c)-\beta_1 c\over \beta(1+c)}\big),\quad \ln {\theta_o\over \theta'}\geq \ln {\beta \over \beta(1+c)-\beta_1 c}.$$
Now
$$\ln{\theta_o\over \theta'} - c \int_\X \ln {g\over f} d\ba\geq \ln {\beta \over \beta(1+c)-\beta_1 c} -c \ln {\beta_1\over \beta}>0.$$
To see that the last inequality holds note that it is equivalent to
$$\beta^{c+1}> \beta_1^c(\beta(1+c)-\beta_1 c).$$
The function $\beta\mapsto \beta^{c+1} - \beta_1^c(\beta(1+c)-\beta_1 c)$ is strictly positive at ${c\over c+1}$ and at 1 (both statements hold for any $\beta_1\in [0,1]$ and $c>0$). The function has unique minimum at $\beta_1$, where it equals 0 and hence it is strictly positive elsewhere.\end{proof}

\subsubsection{Proof of the main theorem when $\lambda=0$.}
%

\begin{proof}[Proof of Theorem~\ref{thm1a}]
\begin{enumerate}
\item
For any $q\in \P_1$, $G(q)-D(\ba\|q)={1\over c}\ln(\langle w,q
\rangle)-D(\ba\|q)={1\over c}F(q)$. Since for $q\not \in \P_1$, $D(\ba \|q)=\infty$, but $G(q)\geq 0$, it holds that
$$\sup_{q\in \P}[G(q)-D(\ba\|q)]=\sup_{q\in \P_1} {1\over c}F(q)={1\over c}F(q^*).$$
Indeed,
if \eqref{marta} holds, the last equality follows from Lemma \ref{lemma1} and Lemma \ref{lemma3}. On the other hand, if \eqref{marta} fails, then we apply Lemma \ref{lemma2} and Lemma \ref{lemma4} instead. By assumption $x_o$ is unique
so
$q^*\in \P_1$ is the unique maximizer
of $F(q)$. Therefore
$$I(q)=\left\{
         \begin{array}{ll}
          {1\over c}\big( F(q^*)-F(q)\big), & \hbox{if $q\in \P_1$;} \\
           \infty, & \hbox{else.}
         \end{array}
       \right.
$$
The LDP implies: for any closed set $C$:
\begin{equation}\label{ldp}
\lim\sup_{n}{1\over cn}\ln \bp(C)\leq -\inf_{q\in
C}I(q).\end{equation}
{
By Lemma 6.2.13 in \cite{LDP},
$$
D(\ba\|q)
=\sup_{g\in C_b} [\langle g ,\ba\rangle-\ln \langle e^g,q\rangle].$$
For every $g$, the function $q\mapsto \langle g,\ba\rangle-\ln \langle e^g,q\rangle$ is continuous and so their supremum $D(\ba\|\cdot)$ is lower semicontinuous. Therefore
$F$ is a upper semicontinuous function.
Now  take $C=B^c$, where $B$ is an open ball (with respect to the Prokhorov metric) containing $q^*$. Thus, $C$ is compact and so $\sup_{q \in C}F(q)=F(q_o)$ for some $q_o \in C$.  Moreover
we know that $F(q) <F(q^*)$ for every $q \neq q^*$, in particular $F(q_o)<F(q^*)$ since $q^* \not \in C$. So we have shown that}
\begin{equation}\label{j}
\sup_{q\in C}F(q)<F(q^*).
\end{equation}
From \eqref{j} we have $\inf_{q\in C}I(q)>0$ and so from \eqref{ldp}, it follows that
$\bp(C)\to 0$ exponentially fast. This means $\bp \Rightarrow
\delta_{q^*}$.
\item
When $\lambda=0$, $w_n=w$ and so $r_{q,n}=r_q$. By Theorem \ref{thmP} and Corollary \ref{cor1},
 the limit process $X_1,X_2,\ldots$ exists and it is  the i.i.d process, where $X_i\sim r_{q^*}$.
When \eqref{marta} holds, then
$$
r_{q^*}(A)=
\frac1{\langle w,f\rangle}\int_A w(x)f(x)\ba(dx),\quad A\in \B(\X),
$$
where $f$ is as in \eqref{tihedus2}. When \eqref{marta} fails, then (recall \eqref{eq:thetao})
$$
r_{q^*}(A)={1\over \theta_o}\Big(\int_A w(x)f(x)\ba(dx)+(1-\beta) w(x_o)
\delta_{x_o}(A)\Big),\quad A\in \B(\X),
$$
where $f$ is as in \eqref{tihedus3} and  $\beta=\int_\X f d\ba$. Since $\theta_o=\frac{w(x_o)}{1+c}$,
we see that when \eqref{marta} fails then
$\bar \a(\{x\colon w(x)= w(x_o)\})=0$,
hence the proportion of $x_o$-types in the limit population equals
 $r_{q^*}(\{x_o\})=(1-\beta)(1+c)$.
\item
The convergence of $Q_n$ follows from Theorem \ref{thmQ}.
\end{enumerate}
\end{proof}\\\\
\subsection{The case $\lambda \in (0,1)$.}
\label{main-lambda01appendix}
This section contains the technical result that we need to prove the main theorems of Section~\ref{main-lambda01}.
\subsubsection{A generalization of Theorem \ref{thmFeng}}
In Section~\ref{main-lambda01}, we stated Theorem~\ref{thmFeng2} which is a generalization of Theorem~\ref{thmFeng} for uniform convergence. Our first task now is to prove this theorem.
The proof relies on
Theorems~\ref{thm:FengLDP}, \ref{thm:var} and
\ref{thm-abi}
below.
\begin{theorem}\label{thm:FengLDP}
 {\rm (\cite{Feng}, Theorem 9.2)} Assume $\X$ is compact. Then $\pi_n$ satisfies the LDP with speed $m_n^{-1}$ and rate function $H(q)=D(\ba\|q)$.\end{theorem}
Here the rate function $H(q)$ equals $D(\ba\|q)$.
{As argued in the proof of Theorem \ref{thm1a}, $D(\ba\|\cdot)$
is lower semicontinuous} and so the level set $\{q: D(\ba\|q)\leq \alpha\}$ is closed for every $\alpha>0$.
Recall that a rate function is good if all level sets are compact; hence if $\X$ is compact, then the rate function
$H(q):=D(\ba\|q)$ is good.
\begin{theorem}\label{thm:var} {\rm (Varadhan lemma (\cite[Theorem B.1]{Feng})} Assume $\X$ is compact and $\pi_n$ satisfies the  LDP with speed $m_n$ and good rate function $H$. Let $G_n$ and $G$ be a family of continuous functions on $\P$  satisfying
$\sup_q |G_n(q)-G(q)|\to 0$
Then
\begin{equation}\label{vard}
\lim_n {1\over m_n} \ln \int_\P \exp[m_n \cdot G_n(q)]\pi_n(dq)=\sup_q \big(G(q)-H(q)\big).
\end{equation}
\end{theorem}
\begin{theorem}\label{thm-abi} {\rm (\cite{Feng}, Theorem B.6))}  Assume $\X$ is compact and $\bp$ is such that
\begin{align*}
\lim_{\delta\to 0} \lim\sup_n {1\over m_n}  \ln \bp \big(B(q,\delta)\big)=\lim_{\delta\to 0}\lim\inf_{n} {1\over m_n} \ln \bp \big(B^o(q,\delta)\big)=-I(q),\end{align*}
where
$B(q,\delta):=\{p:d(p,q)\leq \delta\},\quad B^o(q,\delta):=\{p:d(p,q)< \delta\}$ and $I$ is a good rate function. Then $\bp$ satisfies the LDP with rate function $I$ and speed $m_n^{-1}$.
\end{theorem}
\begin{proof}[Proof of Theorem~\ref{thmFeng2}]
Since $\sup_q |G_n(q)-G(q)|\to 0$, for every $\e_1>0$ there exists $n_1$ so that $|G_n(q)- G(q)|\leq \e_1$, whenever $n>n_1$. Fix $q$ and $\e_2$ and take $\delta>0$ so small $|G(q)-G(p)|\leq \e_2 $ for every $p\in B(q,\delta)=:B$. Estimate
\begin{align*}
\int_B e^{m_n G_n(p)}\pi_n(dp)\leq \int_B e^{m_n \big(G(p)+\e_1\big)}\pi_n(dp)\leq \int_B e^{m_n \big(G(p)+\e_1+\e_2\big)}\pi_n(dp)\leq e^{m_n \big( G(q)+\e_1+\e_2\big)}\pi_n(B).\end{align*}
Then by Theorem \ref{thm:FengLDP}
\[
\begin{split}
\lim\sup_n{1\over m_n}\ln \Big(\int_B e^{m_n G_n(p)}\pi_n(dp)\Big)
& \leq \big( G(q)+\e_1+\e_2\big)+\lim\sup_n \big( \ln \pi_n(B)\big)\\
& \leq  G(q)+\e_1+\e_2 -\inf_{p\in B}H(p).
\end{split}
\]
Similarly, with $B^o=B^o(q,\delta)$,
\[
\begin{split}
\lim\inf_n{1\over m_n}\ln \Big(\int_B e^{m_n G_n(p)}\pi_n(dp)\Big)
& \geq \big( G(q)-\e_1-\e_2\big)+\lim\inf_n \big( \ln \pi_n(B^o)\big)\\
& \geq  G(q)-\e_1-\e_2 -\inf_{p\in B^o}H(p).
\end{split}
\]
Since
$$Z_n= \int_\P e^{m_n G_n(p)}\pi_n(dp),$$
by Theorem \ref{thm:var},
$${1\over m_n} \ln Z_n\to \sup_p\big(G(p)-H(p)\big).$$
Therefore
\begin{align*}
&\lim\sup_n{1\over m_n}\ln \bp (B)\leq G(q)+\e_1+\e_2 -\inf_{p\in B}H(p)-\sup_p \big(G(p)-H(p)\big)\\
&\lim\inf_n{1\over m_n}\ln \bp (B^o)\geq G(q)-\e_1-\e_2 -\inf_{p\in B^o}H(p)-\sup_p \big(G(p)-H(p)\big)
\end{align*}
Let $\delta\to 0$. Then $\e_2(\delta)$ goes to 0 and $\lim_n -\inf_{p\in B}H(p)=-H(p)$, because $-H$ is upper semi-continuous and so
$$
\lim_{\delta\to 0}-\inf_{B} H(p)=-H(q)
$$
because  $-\inf_{p\in B(q,\delta_n)}H(p)\geq -H(q)$ and so, taking $\delta_n\to0$,
 $\lim\inf_n -\inf_{p\in B(q,\delta_n)}H(p)\geq -H(q)$, but when $p_n\Rightarrow q$ is such that
$\lim\sup_n -\inf_{p\in B(q,\delta_n)}H(p)=\lim\sup_n-H(p_n)$, then by USC, it holds $\lim\sup_n -\inf_{p\in B_n}H(p) \leq -H(q)$. The same holds when $B$ is replaced by $B^o$.
\\
Therefore
\begin{align*}
&\lim_{\delta\to 0}\lim\sup_n{1\over m_n}\ln \bp (B)\leq G(q)+\e_1 -H(q)-\sup_p \big(G(p)-H(p)\big)\\
&\lim_{\delta\to 0}\lim\inf_n{1\over m_n}\ln \bp (B^o)\geq G(q)-\e_1-H(q)-\sup_p \big(G(p)-H(p)\big)
\end{align*}
Since $\e_1$ was arbitrary, we see that the assumptions of Theorem \ref{thm-abi} hold, therefore $\bp$ satisfies the LDP with speed $m_n$ and rate function $I$.
\end{proof}
\subsubsection{Optimizing the function $F$}
\paragraph{The objective function $F$.} We now define
$$F(q)=-\langle \phi, q\rangle - c  H(q),\quad q\in \P_1.$$
When $q\in \P_0$, then $\exists g={d q \over d \ba}$ and then
$$F(q)=-\int_\X \phi g d\ba+c \int_\X \ln g d\ba=- \langle \phi,q \rangle+c \int_\X \ln g d\ba=F(g).$$
\subsubsection{Maximizing $F$: \eqref{marta2} holds}
\begin{lemma}\label{lemma1a}
Assume that \eqref{marta2} holds. Then $F$ has a unique maximizer over $\P_1$, which is the measure with density \eqref{tihedus4} with respect
to $\bar\alpha$.\end{lemma}
\begin{proof}Let $f$ be the density \eqref{tihedus4}. Again, we split the maximization: over $\cal{F}$ and over $\P_1\setminus \P_0$. At first, we show that for every $g\in {\cal F}$ such that  $g\ne f$  $\ba$-a.s., it holds
\begin{equation}
F(f)-F(g)=- \theta + c\int_\X \ln f d\ba  + \theta' - c\int_\X \ln g d\ba = {\theta' - \theta} - c \int_\X \ln {g\over f} d\ba >0,
\end{equation}
where  $\theta'=\langle \phi,g \rangle$. When $\int_\X \ln g d\ba=-\infty$, then inequality strictly holds, otherwise observe that all integrals above are finite and since
$$\int_\X {g\over f} d\ba={\theta'-\theta+c\over c},$$
by Jensen inequality we get
$$- c \int_\X \ln {g\over f} d\ba > - \ln \Big[{\theta'-\theta+c\over c}\Big]^c,$$
where the strict inequality follows from assumption $f\neq g$ $\ba$ - a.s.. Therefore, it suffices to show that the L.H.S.~is nonnegative; to this aim, note that
$$
{\theta' -  \theta}-   \ln \Big[{\theta'-\theta+c\over c}\Big]^c=
c \Big (\frac{\theta' -  \theta}{c}-   \ln \Big[\frac{\theta'-\theta}{c}+1\Big] \Big )\geq 0.
$$
We now take $q\in \P_1\setminus \P_o$. Let $h={d\ba \over dq}$, $g=1/h$, $\beta_1=\int_\X g d\ba \equiv q(h>0)$, $\theta'=\langle \phi,q \rangle\geq \langle \phi,g \rangle + \phi_o(1-\beta_1)$. As previously, we obtain via Jensen's inequality
\begin{equation}\label{nojah}
F(f)-F(q)=- \theta + c\int_\X \ln f d\ba  + \theta' - c\int_\X \ln g d\ba  \geq {\theta' - \theta} - c \ln \big({\langle \phi, g\rangle +(c-\theta)\beta_1\over c}\big)
\end{equation}
and it remains to show that
$$\ln \big({\langle \phi, g\rangle +(c-\theta)\beta_1 \over c}\big)\leq {\theta' - \theta\over c}.$$
Since $\langle \phi, g\rangle\leq \theta'-\phi_o(1-\beta_1)$ and $c-\theta \geq -\phi_o$  we obtain that
\begin{align*}\label{viimane}\ln \big({\langle \phi, g\rangle +(c-\theta)\beta_1 \over c}\big)&\leq \ln \big({\theta' +(c-\theta)\beta_1 -\phi_o(1-\beta_1)\over c} \big)\leq \ln \big({\theta' +(c-\theta)\beta_1 + (c-\theta)(1-\beta_1)\over c} \big)\\
&\leq \ln \big(1+{\theta'-\theta \over c}\big)
\leq  {\theta' - \theta\over c}.\end{align*}
Now it remains to argue that at least one of the inequalities is strict. The Jensen inequality is an equality only if $\beta_1f=g$ and in this case
$F(f)-F(q)=\theta'-\theta - c\ln \beta_1$. Since for $g=\beta_1 f$, it holds  $\theta'\geq \beta_1 \theta+\phi_o(1-\beta_1)$ we obtain that $\theta'-\theta \geq (1-\beta_1)(\phi_o-\theta)\geq -c(1-\beta_1)>c\ln \beta_1$, and so
$F(f)-F(q)=\theta'-\theta - c\ln \beta_1>0$.
\end{proof}
\subsubsection{Maximizing $F$: \eqref{marta2} fails}
Remember the definition of $q^*$ given in
equation~\eqref{q-def2}.
Now
\begin{equation}\label{eq:5.29.5}
\langle \phi, q^*\rangle = \langle \phi, f\rangle +(1-\beta)\phi_o=c+\phi_o\beta+(1-\beta)\phi_o=c+\phi_o=:\theta_o.
\end{equation}
Since \eqref{marta} fails then $\ba(x_o)=0$, which in turn implies $q^*\in \P_1$, and then $F(q^*)=-\phi_o+c\int_\X \ln f d \ba.$
\begin{lemma}\label{lemma2a} Let \eqref{marta2} fail, $x_o$ be the unique minimizer of $\phi$. Then for every $q\in \P_1$ such that $q\ne q^*$, it holds $F(q)<F(q^*)$.\end{lemma}
\begin{proof} Again, we start with maximizing over $\F$. Let $g\in \F$. Then, by Jensen Inequality
\begin{align*}
F(q^*)-F(g)&=- \theta_o + c\int_\X \ln f d\ba  + \theta' - c\int_\X \ln g d\ba = {\theta' - (\phi_o+c)} - c \int_\X \ln {g\over f} d\ba \\
           &\geq  {\theta' - (\phi_o+c)} - c\ln \Big[{\theta'-\phi_o\over c}\Big]\geq 0,
\end{align*}
because for every $x$, $x-1\ge \ln x$. The Jensen inequality is an equality, when $\beta g=f$ and then $F(q^*)-F(g)=\theta'-\theta_o+c\ln \beta$. Then also $\theta'=\langle \phi,g \rangle=\beta^{-1} \langle \phi,f \rangle=c/\beta + \phi_o$ and, therefore, $$F(q^*)-F(g)=\theta'-c-\phi_o+c\ln \beta=c/\beta + \phi_o -c-\phi_o+c\ln \beta = c({1\over \beta}-1)-c\ln {1\over \beta}>0.$$
We now take $q\in \P_1\setminus \P_0$. Let, again, $h={d\ba \over dq}$, $g=1/h$, $\beta_1=\int_\X g d\ba$, $\theta'=\langle \phi,q \rangle\geq \langle \phi,g \rangle + \phi_o(1-\beta_1)$ and $\theta_o=\phi_o+c$. The inequality \eqref{nojah} now reads
\begin{equation}\label{nojah2}
F(q^*)-F(q)=- \theta_o + c\int_\X \ln f d\ba  + \theta' - c\int_\X \ln g d\ba  \geq {\theta' - \theta_o} - c \ln \big({\langle \phi, g\rangle -\phi_o\beta_1\over c}\big)
\end{equation}
and
$$  \ln \big({\langle \phi, g\rangle -\phi_o\beta_1\over c}\big)\leq \ln \big({\theta' - \phi_o(1-\beta_1)-\phi_o\beta_1\over c}\big)=\ln \big({\theta' - \phi_o\over c}\big)=\ln \big(1+{\theta'-\theta_o\over c}\big)\leq {\theta'-\theta_o\over c}.$$
The Jensen inequality is an equality when $g={\beta_1\over \beta}f$. When $\beta=\beta_1$, then $f=g$ and $F(q^*)-F(q)=\theta'-(\phi_o+c)$. Since $q\ne q^*$, it follows that $\theta'> \langle \phi, f \rangle + (1-\beta)\phi_o=c+\phi_o$ and so $F(q^*)>F(q)$ (analogously as in the proof of Lemma~\ref{lemma4}). Consider now the case $\beta_1\ne \beta$. Then
$$F(q^*)-F(q)=\theta'-(\phi_o+c)+c (\ln \beta-\ln \beta_1).$$ Since now $$\theta'\geq \beta_1/\beta \langle \phi,f \rangle +\phi_o(1-\beta_1)=\beta_1/\beta(c+\beta\phi_o)+\phi_o(1-\beta_1)={\beta_1\over \beta}c+\phi_o,$$
we obtain that $F(q^*)-F(q)\geq {\beta_1\over \beta}c+\phi_o-(\phi_o+c)-c\ln ({\beta_1\over \beta})=c({\beta_1\over \beta}-1)-c\ln ({\beta_1\over \beta})>0$.\end{proof}
\subsubsection{Proof of the main theorem when $\lambda \in (0,1)$.}
\begin{proof}[Proof of Theorem~\ref{thm1b}]
\begin{enumerate}
\item
The proof $\bp\Rightarrow \delta_{q^*}$ is exactly as in Theorem \ref{thm1a}, just instead of Theorem \ref{thmFeng}, Theorem \ref{thmFeng2} should be used.
\item
Note that $w_n(x)\to 1$ and
$\sup_x|w_n(x)-1|=1-\exp[-{{\phi}_o\over n^{\lambda}}]\to 0$,
 since $\phi$ is continuous $\phi$ and $\X$ compact.
All assumptions of Corollary \ref{cor1} are fulfilled with $r_{q}=q$.
 By Theorem \ref{thmP}, the limit  process $X_1,X_2,\ldots$, exists, and it is an i.i.d. process, where $X_i\sim q^*$.
\item By Theorem \ref{thmQ}, $Q_n\Rightarrow \delta_{q^*}$
\end{enumerate}
\end{proof}

\end{document}